\documentclass{amsart}
\usepackage{amssymb,color,xcolor}
\usepackage{eucal}
\usepackage{hyperref}
\usepackage{ifpdf}   
\theoremstyle{plain}
\newtheorem{theorem}{Theorem}
\newtheorem{corollary}{Corollary}

\newtheorem{lemma}{Lemma}

\theoremstyle{definition}

\theoremstyle{remark}

\newtheorem{remark}{Remark}
\numberwithin{equation}{section}
\newcommand{\e}{\epsilon}
\newcommand{\s}{\sigma}

\newcommand{\R}{\mathbb R}
\newcommand{\N}{\mathbb N}

\newcommand{\C}{\mathbb C}

\newcommand{\Rn}{\mathbb R^n}

\newcommand{\rst}[1]{\ensuremath{{\mathbin\upharpoonright}%
\raise-.5ex\hbox{$#1$}}}
 \newcommand{\sgn}{\mathop{\mathrm{sgn}}}
\def\p{\partial}

\def \e {\epsilon}

\def \om {\Omega}

\def \s{\sigma}
\def \ll{\label}
\begin{document}

\title[Observations from measurable sets]{Observations
from measurable sets and applications}
\author{Luis Escauriaza}
\address[Luis Escauriaza]{Universidad del Pa{\'\i}s Vasco/Euskal Herriko Unibertsitatea\\Dpto. de Matem\'aticas\\Apto. 644, 48080 Bilbao, Spain.}
\email{luis.escauriaza@ehu.es}
\thanks{The first two authors are supported  by Ministerio de Ciencia e Innovaci\'on grant MTM2011-2405. The last author is supported by the National Natural Science Foundation of China under grant
 11171264.}
\author{Santiago Montaner}
\address[Santiago Montaner]{Universidad del Pa{\'\i}s Vasco/Euskal Herriko
Unibertsitatea\\Dpto. de Matem\'aticas\\Apto. 644, 48080 Bilbao, Spain.}
\email{santiago.montaner@ehu.es}
\author{Can Zhang}
\address[Can Zhang]{School of Mathematics and Statistics, Wuhan University, Wuhan, China}
\email{zhangcansx@163.com}
\begin{abstract}
We find new quantitative estimates on the space-time analyticity of solutions to linear parabolic equations with  time-independent coefficients and apply them to obtain observability inequalities for its solutions over measurable sets.
\end{abstract}
\maketitle
\section{Introduction}\label{S:1}

Mixing up ideas developed in \cite{W1}, \cite{AE} and \cite{PW1}, it was shown in \cite{AEWZ} that the heat equation over bounded domains $\Omega$ in $\Rn$ can be null controlled at all times $T>0$ with interior and bounded controls acting over space-time measurable sets $\mathcal D\subset\Omega\times (0,T)$ with positive Lebesgue measure, when $\Omega$ is a Lipchitz polyhedron or a $C^1$ domain in $\Rn$. \cite{AEWZ} also established the boundary null-controllability with bounded controls over measurable sets $\mathcal J\subset\partial\Omega\times (0,T)$ with positive surface measure.

In this work we explain the techniques necessary  to apply the same methods in \cite{AEWZ} in order to obtain the interior and boundary null controllability of  some higher order or non self-adjoint parabolic evolutions with time-independent analytic coefficients over analytic domains $\Omega$ of $\Rn$ and with bounded controls acting over measurable sets with positive measure. We also show the null-controllability with controls acting over  possibly  different measurable regions over each component of the Dirichlet data of higher order parabolic equations or over each component of the solution to second order parabolic systems; both at the interior and at the boundary. Finally, we show that the same methods imply the null-controllability of some not completely uncoupled parabolic systems with bounded interior controls acting over only one of the components of the system and on measurable regions.

We explain the technical details for parabolic higher order equations with constant coefficients and for second order systems with time independent analytic coefficients. We believe that this set of examples will make it clear to the experts that the combination of the methods in \cite{W1}, \cite{AE}, \cite{PW1} with others here imply  analog results to those in \cite{AEWZ} for  parabolic evolutions associated to possibly  non self-adjoint higher order elliptic equations or second order systems with time independent analytic coefficients over analytic domains: existence of bounded null-controls acting over measurable sets and the uniqueness and bang-bang property of certain optimal controls.

Throughout the work $0<T\le 1$ denotes a positive time, $\Omega\subset\Rn$ a is bounded domain with analytic boundary $\p\Omega$, $\nu$  is the  exterior unit normal vector to the boundary of $\Omega$ and $d\sigma$ denotes surface measure on $\partial\Omega$, $B_R(x_0)$ stands for the ball centered at $x_0$
and of radius $R$, $B_R=B_R(0)$. For measurable sets $\omega\subset\Rn$ and $\mathcal{D}\subset\Rn\times (0,T)$, $|\omega|$ and $|\mathcal{D}|$ stand for the Lebesgue measures of the sets; for  measurable sets $\Gamma\subset\p\Omega$ and $\mathcal{J}$ in $\partial\Omega\times (0,T)$, $|\Gamma|$ and $|\mathcal{J}|$ denote respectively their surface measures in $\p\Omega$ and $\partial\Omega\times \R$. $|\alpha|=\alpha_1+\dots+\alpha_\ell$, when $\alpha=\left(\alpha_1,\dots,\alpha_\ell\right)$ is a $\ell$-tuple in $\N^{\ell}$, $\ell\ge 1$.

To describe the analyticity of the boundary of $\Omega$ we assume that there is some $\delta>0$ such that for each $x_0$ in $\partial\Omega$ there is, after a translation and rotations, a new coordinate system (where $x_0=0$) and a real analytic function $\varphi : B_{\delta}'\subset \R^{n-1}\longrightarrow\R$ verifying
\begin{equation}\label{E:descripcionfrontera}
\begin{aligned}
\varphi(0'&)=0,\ |\partial_{x'}^{\alpha}\varphi(x')|\le |\alpha|! \,\delta^{-|\alpha|-1}\, ,\ \text{when}\ x'\in B_{\delta}',\ \alpha\in \N^{n-1},\\
&B_{\delta}\cap\Omega=B_\delta\cap\{(x',x_n): x'\in B_{\delta}',\  x_n>\varphi(x')\},\\
&B_{\delta}\cap\partial\Omega=B_\delta\cap\{(x',x_n): x'\in B_{\delta}',\ x_n=\varphi(x')\}.
\end{aligned}
\end{equation}

  The existence of the bounded null-controls acting over the measurable sets for the set of examples follows by standard duality arguments (cf. \cite{DRu} or \cite{Lions1}) from the following list of observability inequalities.
\begin{theorem}\label{6131}
Let $\mathcal{D}\subset\Omega\times(0,T)$ be a measurable set with positive measure and $m\ge 1$. Then, there is a constant $N=N(\Omega,T,m,\mathcal{D},\delta)$ such that the inequality
\begin{equation*}
\|u(T)\|_{L^2(\Omega)}\leq
N\int_{\mathcal{D}}|u(x,t)|\;dxdt
\end{equation*}
holds for all solutions $u$ to
\begin{equation}\label{E: 6081}
\begin{cases}
\p_t u+(-1)^m\Delta^m u=0,\ &\text{in}\ \Omega\times(0,T),\\
u=\nabla u=\dots=\nabla^{m-1}u=0,\ &\text{on}\ \p\Omega\times(0,T),\\
u(0)=u_0,\ &\text{in}\ \Omega,
\end{cases}
\end{equation}
with $u_0$ in $L^2(\Omega)$.
\end{theorem}
\begin{remark}\label{102c1}
The constant $N$ in Theorem \ref{6131} is of the form $e^{N/T^{1/(2m-1)}}$with $N=N(\Omega,|\omega|,\delta)$, when $\mathcal{D}=\omega\times(0,T)$, $0<T\leq1$ and $\omega\subset\Omega$ is a measurable set. The later is consistent  with the case of the heat equation \cite{FernandezZuazua}.
\end{remark}
The second and third are two boundary observability inequalities over measurable sets for the higher order evolution \eqref{E: 6081}. The first over a general measurable set and the second over two possibly different measurable sets with the same projection over the time $t$-axis. To simplify, we give the details only  for the evolution associated to  $\Delta^2$.
\begin{theorem}\label{6106}
Assume that $\mathcal{J}\subset\p\Omega\times(0,T)$ is a measurable set with positive surface measure in $\p\Omega\times(0,T)$. Then, there is $N=N(\Omega,\mathcal{J},T,\delta)$ such that the inequality
\begin{equation}\label{E: firatobservanbi}
\|u(T)\|_{L^2(\Omega)}
\leq N\int_{\mathcal{J}}|\tfrac{\p\Delta u}{\p\nu}(x,t)|+|\Delta u(x,t)|\;d\sigma dt,
\end{equation}
holds for all  solutions $u$ to
\begin{equation}\label{6084}
\begin{cases}
\p_t u+\Delta^2 u=0,\ &\text{in}\ \Omega\times(0,T),\\
u=\frac{\partial u}{\partial\nu}=0,\ &\text{on}\ \p\Omega\times(0,T),\\
u(0)=u_0,\ &\text{in}\ \Omega,
\end{cases}
\end{equation}
with  $u_0$ in $L^2(\Omega)$.
\end{theorem}
\begin{remark}
When $\mathcal{J}=\Gamma\times(0,T)$, $0<T\leq1$ and $\Gamma\subset\p\Omega$, the constant $N$ in Theorem \ref{6106} is of the form $e^{N/T^{1/3}}$ with $N=N(\Omega,|\Gamma|,\delta)$.
\end{remark}
\begin{theorem}\label{614g3}
Assume that $E\subset(0,T)$ is a measurable set with positive measure and that $\Gamma_i\subset \p\Omega$, $i=1,2$, are measurable sets with positive surface measure. Then, there is $N=N(\Omega,|\Gamma_1|,|\Gamma_2|,E,\delta)$ such that the inequality
\begin{equation*}
\|u(T)\|_{L^2(\Omega)}
\leq N\int_{E}\|\tfrac{\p\Delta u}{\p\nu}(t)\|_{L^1(\Gamma_1)}+\|\Delta u(t)\|_{L^1(\Gamma_2)}
\;dt,
\end{equation*}
holds for all  solutions $u$ to \eqref{6084}.
\end{theorem}

\begin{remark}\label{R:44}
We do not know if the sets $\Gamma_1\times E$ and $\Gamma_2\times E$ can be replaced by general measurable sets $\mathcal J_i\subset\partial\Omega\times (0,T)$, $i=1,2$.
\end{remark}
Now we consider the evolutions associated with strongly coupled second order {\it time independent} parabolic systems with a possible non self-adjoint structure, as the second order system
\begin{equation}\label{anas}
\begin{cases}
\partial_t\mathbf u-\mathbf L\mathbf{u}=0,\ &\text{in}\ \Omega\times(0,T),\\
\mathbf{u}=0,\ &\text{on}\ \partial\Omega\times(0,T),\\
\mathbf{u}(0)=\mathbf u_0,\ &\text{in}\ \Omega,
\end{cases}\quad\quad \text{with}\ \mathbf L=(L^1,\dots\mathbf, L^\ell),
\end{equation}
with
\begin{equation*}
L^\xi\mathbf{u}=\partial_{x_i}(a_{ij}^{\xi\eta}(x)\partial_{x_j}
u^\eta)+b_{j}^{\xi\eta}(x)\partial_{x_j}u^\eta+c^{\xi\eta}(x)u^{\eta},
\;\;\xi=1,\dots,\ell,
\end{equation*}
and $\mathbf u_0$ in $L^2(\Omega)^\ell$. Here, $\mathbf{u}$ denotes the vector-valued function $(u^1,\dots,u^{\ell})$ and  the summation convention of repeated indices is understood. We assume that $a_{ij}^{\xi\eta}$, $b_{j}^{\xi\eta}$ and $c^{\xi\eta}$ are analytic functions over $\overline\Omega$, i.e., there is $\delta>0$ such that
\begin{equation}\label{E: condicion3}
|\p_x^\gamma a_{ij}^{\xi\eta}(x)|+ |\p_x^\gamma b_{j}^{\xi\eta}(x)|+ |\p_x^\gamma c^{\xi\eta}(x)|\leq \delta^{-|\gamma|-1}|\gamma|!,\ \text{for all}\ \gamma\in\mathbb{N}^n\ \text{and}\ x\in\overline\Omega,
\end{equation}
 and only requires  that the higher order terms of the system \eqref{anas} have a self-adjoint structure; i.e.
 \begin{equation}\label{E: condicon5}
 a_{ij}^{\xi\eta}(x)=a_{ji}^{\eta\xi}(x),\ \text{for all}\ x\in\overline\Omega,\ \xi,\eta=1,\dots,\ell,\ i,j=1,\dots,n,
 \end{equation}
 together with the strong ellipticity condition
\begin{equation}\label{E: condici�n4}
\sum_{\xi,\eta,i,j}a_{ij}^{\xi\eta}(x)\zeta_{i}^\xi
\zeta_j^\eta\geq \delta\sum_{i,\xi}|\zeta_i^\xi|^2,\ \text{for all}\ \mathbf{\zeta}
=(\zeta_i^\xi)\ \text{in}\ \mathbb{R}^{n\ell}\ \text{and}\ x\in\overline\Omega.
\end{equation}
The results described below also hold when the higher order coefficients of the system verify \eqref{E: condicon5} and the weaker Legendre-Hadamard condition \cite[p. 76]{Giaquinta},
\begin{equation}\label{E:Hadamard2}
\sum_{i, j, \xi,\eta}a^{\xi\eta}_{ij}(x)\varsigma_i\varsigma_j\vartheta^\xi\vartheta^{\eta}\ge\delta |\varsigma|^2 |\vartheta|^2,\ \text{when}\ \varsigma\in \R^{n}, \vartheta\in\R^\ell,\,  x\in\Rn,
\end{equation}
in place of \eqref{E: condici�n4}.
Recall that the Lam\'e system of elasticity
\begin{equation*}\label{E:segundosistema}
\nabla\cdot\left(\mu(x)\left(\nabla \mathbf u+\nabla \mathbf u^t\right)\right) +\nabla\left(\lambda(x)\nabla\cdot \mathbf u\right),
\end{equation*}
with $\mu\ge\delta$, $\mu+\lambda \ge 0$ in $\Rn$, $\ell=n$ and
\begin{equation*}
a^{\xi\eta}_{ij}(x)=\mu(x)(\delta_{\xi\eta}\delta_{ij}+\delta_{i\eta}\delta_{j\xi})+\lambda(x)\delta_{j\eta}\delta_{\xi i},
\end{equation*}
are examples of systems verifying \eqref{E:Hadamard2}.

The observability inequalities related to parabolic second order systems are as follows. The first is an interior observability inequality with possibly different measurable interior observation regions for each component of the system but with the same projection over the time $t$-axis.

\begin{theorem}\label{nball}
Let $E\subset(0,T)$ be a measurable, $|E|>0$ and $\omega_\eta\subset\Omega$, $\eta=1,\dots,\ell$, be  measurable with $|\omega_\eta|\geq \omega_0$, $\eta=1,\dots,\ell$, for some $\omega_0>0$. Then, there is $N=N(\Omega,T,E,\omega_0,\delta)$ such that the inequality
\begin{equation*}
\|\mathbf{u}(T)\|_{L^2(\Omega)^\ell}\leq
N\int_{E}\sum_{\eta=1}^{\ell}\|u^\eta(t)
\|_{L^1(\omega_\eta)}\;dt
\end{equation*}
holds for all solutions $\mathbf{u}$ to \eqref{anas}.
\end{theorem}
\begin{remark}\label{R:45}
We do not know if the sets $\omega_\eta\times E$, $\eta=1,\dots,\ell$, can be replaced by different and more general measurable sets $\mathcal D_\eta\subset\Omega\times (0,T)$.
\end{remark}
The second is a boundary observability inequality over possibly different measurable sets with the same projection over the time $t$-axis for each component of the system and the third, a  boundary observability over a general measurable subset of $\partial\Omega\times (0,T)$.
\begin{theorem}\label{6090}
Let $E\subset(0,T)$ be a measurable set with a positive measure and $\gamma_\eta\subset \p\Omega$, $\eta=1,\dots, \ell$, be measurable sets with $\min_{\eta=1,\dots,\ell}{|\gamma_\eta|}\ge \gamma_0$, for some $\gamma_0>0$. Then, there is $N=N(\Omega,E,T, \gamma_0,\delta)\geq1$ such that the inequality
\begin{equation*}
\|\mathbf{u}(T)\|_{L^2(\Omega)^\ell}\leq
N\int_{E}\sum_{\eta=1}^\ell\|\tfrac{\partial u^\eta}{\partial\nu}(t)\|_{L^1(\gamma_\eta)}\,dt
\end{equation*}
holds for all solutions $\mathbf u$ to \eqref{anas}. Here $\frac{\partial\mathbf u}{\partial \nu}=\big(\frac{\partial u^1}{\partial\nu},\dots,\frac{\partial u^\ell}{\partial\nu}\big)$ with $
\frac{\partial u^\eta}{\partial\nu}
\triangleq a_{ij}^{\eta\xi}\partial_{x_j}u^\xi\nu_i,
$
for $\eta=1,\dots,\ell.$
\end{theorem}
\begin{remark}\label{R:47}
We do not know if the sets $\gamma_\eta\times E$, $\eta=1,\dots,\ell$, can be replaced by different and general measurable sets $\mathcal J_\eta\subset\partial\Omega\times (0,T)$, $\eta=1,\dots, \ell$.
\end{remark}
\begin{theorem}\label{614g2}
Let $\mathcal{J}$ be measurable subset of $\p\Omega\times(0,T)$ with positive measure. Then, there is $N=N(\Omega,\delta,\mathcal{J},T)$ such that the inequality
\begin{equation*}
\|\mathbf{u}(T)\|_{L^2(\Omega)^\ell}
\leq N\int_{\mathcal{J}}\;|\tfrac{\p\mathbf{ u}}{\p\nu}(q,t)|\;d\sigma dt.
\end{equation*}
holds
for all solutions $\mathbf{u}$ to \eqref{anas},
\end{theorem}
With the same methods as for Theorem \ref{6131}
one can also get an observability inequality for \eqref{anas} with observations over general interior measurable
sets.
\begin{theorem}\label{614g1}
Let $\mathcal{D}\subset\Omega\times(0,T)$ be a measurable set with positive measure. Then there is $N=N(\Omega,T,\mathcal{D},\delta)\geq1$ such that the inequality
\begin{equation*}
\|\mathbf{u}(T)\|_{L^2(\Omega)^\ell}\leq
N\int_{\mathcal{D}}|\mathbf{u}(x,t)|\;dxdt,
\end{equation*}
holds for all solutions $\mathbf{u}$ to \eqref{anas}.
\end{theorem}
\begin{remark}
The constant in Theorem \ref{614g1} is of the form $e^{N/T}$ with $N=N(\Omega,\omega,\delta)$, when $\mathcal{D}=\omega\times(0,T)$, $0<T\le 1$ and $\omega\subset\Omega$.
\end{remark}

Finally, the last observability inequality deals with the observation of only one interior component of two coupled parabolic equations over a measurable set (See \cite{WL} for the case of open sets). In particular, we consider the {\it time independent} not completely uncoupled parabolic system
\begin{equation}\label{heats2}
\begin{cases}
\partial_t u-\Delta u+a(x)u+b(x)v=0,\ &\text{in}\;\;\Omega\times(0,T),\\
\partial_t v-\Delta v+c(x)u+d(x)v=0,\ &\text{in}\;\;\Omega\times(0,T),\\
u=0,\;\;v=0,\;\;&\text{on}\;\;\partial\Omega\times(0,T),\\
u(0)=u_0,\;\;v(0)=v_0,\;\;&\text{in}\;\;\Omega,
\end{cases}
\end{equation}
with $a$, $b$, $c$ and $d$ analytic in $\overline\Omega$, $b(\cdot)\neq 0$, somewhere in $\overline\Omega$ and with
\begin{equation*}
|\p_x^\gamma a(x)|+|\p_x^\gamma b(x)|+ |\p_x^\gamma c(x)|+ |\p_x^\gamma d(x)|\leq \delta^{-|\gamma|-1}|\gamma|!,\ \text{for all}\ \gamma\in\mathbb{N}^n\ \text{and}\ x\in\overline\Omega,
\end{equation*}
for some $\delta>0$. Then, we get the following bound.
\begin{theorem}\label{bangbang-onecontrol}
Let $\mathcal D\subset \Omega\times(0,T)$ be a measurable set with  positive measure. Then there is $N=N(\Omega,\mathcal D,T, \delta)$ such that the inequality
\begin{equation*}
\|u(T)\|_{L^2(\Omega)}+
\|v(T)\|_{L^2(\Omega)}
\leq N\int_{\mathcal D}|u(x,t)|\,dxdt,
\end{equation*}
holds  for all solutions $(u,v)$ to \eqref{heats2}
\end{theorem}
\begin{remark}\label{R: 7}
Theorem \ref{bangbang-onecontrol} is still valid when the Laplace operator $\Delta$ in  \eqref{heats2} is replaced by two second elliptic operators $\nabla\cdot(\mathbf A_i(x)\nabla\cdot)$, $i=1,2$, with matrices $\mathbf A_i$ real-analytic, symmetric and positive-definite over $\overline\Omega$. Here, we must make sure that the higher order terms of the system remain uncoupled: a diagonal principal  part. Otherwise, we do not know if such kind  of observability estimates are possible. We believe that generally they are not.
\end{remark}

As far as we know, the observability inequalities for the evolutions \eqref{E: 6081} for $m\ge 2$ and \eqref{anas} have not been proved with Carleman methods; not even when $\mathcal D$, $\mathcal J$, $\Gamma_1$, $\Gamma_2$, $\omega_\eta$, $\gamma_\eta$ are open sets and $E=(0,T)$, cases where the standard techniques to prove Carleman inequalities should make it more feasible. The reasons for these are the difficulties that one confronts when dealing with the calculation and test of the positivity of the commutators associated to the Carleman methods for higher order equations and second order systems.

The method we use relies on the {\it telescoping series method} -  built with ideas borrowed from \cite{Miller2} and first used in \cite{PW1} - and on local observability inequalities for analytic functions over measurable sets: the Lemma \ref{propagation} as in \cite{AE, AEWZ} and a new extension of Lemma \ref{propagation}, the Lemma \ref{higherderivative} below. We use Lemma \ref{higherderivative} in the proof of Theorem \ref{bangbang-onecontrol}.

\begin{lemma}\label{propagation}
Let $\Omega$ be a bounded domain in $\Rn$ and $\omega\subset\Omega$ be a measurable set of positive measure. Let $f$ be an analytic function in $\Omega$ satisfying
$$|\p_x^\alpha f(x)|\leq M\rho^{-|\alpha|}|\alpha|!,\ \text{for}\  x\in\Omega\ \text{and}\ \alpha\in\mathbb{N}^n,$$
for some numbers $M$ and $\rho$. Then, there are $N=N(\Omega,\rho,|\omega|)$ and $\theta=\theta(\Omega,\rho,|\omega|)$, $0<\theta<1$, such that
\begin{equation*}
\|f\|_{L^\infty(\Omega)}\leq NM^{1-\theta}\Big(\text{\rlap |{$\int_{\omega}$}}|f|\,dx\Big)^{\theta}.
\end{equation*}
\end{lemma}
 Lemma \ref{propagation} was first derived in \cite{Vessella}. See also \cite{Nadirashvili2} and \cite{Nadirashvili} for close results. The reader can find a simpler proof of  Lemma \ref{propagation} in \cite[\S 3]{AE}. The proof there is  built with ideas from \cite{Malinnikova}, \cite{Nadirashvili2} and  \cite{Vessella}.

\begin{lemma}\label{higherderivative} Let $\Omega$ be a bounded domain in $\Rn$ and $\omega\subset\Omega$ be a measurable set with positive Lebesgue measure. Let $f$ be an analytic function in $\Omega$ satisfying
\begin{equation*}
|\partial_x^\alpha f(x)|\leq M|\alpha|!\,\rho^{-|\alpha|},\ \text{for}\ \alpha\in\mathbb{N}^n \;\;\text{and}\ x\in \Omega,
\end{equation*}
for some $M>0$ and $0<\rho\le 1$. Then, there are constants $N=N(\Omega,\rho,|\omega|, n)$ and $\theta=\theta(\Omega,\rho,|\omega|)$, $0<\theta<1$, such that
\begin{equation*}
\|\partial^\alpha_x f\|_{L^\infty(\Omega)}\leq |\alpha|!\, (\rho/N)^{-|\alpha|-1}
M^{1-\frac{\theta}{2^{|\alpha|}}}
\Big(\text{\rlap |{$\int_{\omega}$}}|f|\,dx\Big)^{\frac{\theta}{2^{|\alpha|}}}, \ \text{when}\  \alpha\in\N^n.
\end{equation*}
\end{lemma}

To the best of our knowledge, the works  studying the space-time analyticity of solutions to linear parabolic equations or systems with space-time analytic coefficients over analytic domains with zero  Dirichlet lateral data or with other types of zero lateral data \cite{Friedman2,Tanabe, Friedman1,Eidelman,KinderlehrerNirenberg,Komatsu} do not in general state clearly the quantitative estimates on the analyticity of the solutions derived from the methods they use. Likely, the authors were mostly interested in the qualitative behavior.

As far as we understand, the best quantitative bound that one can get for solutions to  \eqref{E: 6081}, \eqref{6084}, \eqref{anas} and \eqref{heats2} with initial data in $L^2(\Omega)$ from the works \cite{Friedman2,Tanabe, Friedman1,Eidelman,KinderlehrerNirenberg,Komatsu} is the following:

 \vspace{0.1cm}
{\it There is $0<\rho\le 1$, $\rho=\rho(m, n,\delta)$ such that
\begin{equation}\label{E: kimalito}
|\partial_x^\alpha\partial_t^p u(x,t)|\le \rho^{-1-\frac{|\alpha|}{2m}-p}|\alpha|!\,\,p!\, t^{-\frac{|\alpha|}{2m}-p-\frac{n}{4m}}\|u_0\|_{L^2(\Omega)}, \forall \alpha\in\N^n, p\in\N,
\end{equation}
where $2m$ is the order of the evolution.}
\vspace{0.1cm}

\noindent  The arguments in \cite{Friedman2,Tanabe, Friedman1,Eidelman,KinderlehrerNirenberg,Komatsu} show that \eqref{E: kimalito} holds when the coefficients of the underlying linear parabolic equation or system are time dependent and satisfy bounds like
\begin{equation}\label{E condiciongeneral}
|\partial_x^\alpha\partial_t^p A(x,t)|\le \delta^{-1-|\alpha|-p}|\alpha|!\,p!\, ,\ \text{for all}\ \alpha\in\N^n,\ p\in\N,\ x\in\overline\Omega\ \text{and}\ t>0,
\end{equation}
for some $0<\delta\le 1$. On the other hand,  there is $\rho=\rho(n,m)$, $0<\rho\le 1$, such that the solution to
\begin{equation*}
\begin{cases}
\partial_t u+(-\Delta)^m u=0,\ &\text{in}\ \Rn\times (0,+\infty),\\
u(0)=u_0,\ &\text{in}\ \Rn,
\end{cases}
\end{equation*}
verifies
\begin{equation}\label{E: quecachondeo}
|\partial_x^\alpha\partial_t^pu(x,t)|\le \rho^{-1-\frac{|\alpha|}{2m}-p}|\alpha|!^{\frac 1{2m}}\,p!\, t^{-\frac{|\alpha|}{2m}-p-\frac{n}{4m}}\|u_0\|_{L^2(\Rn)},
\end{equation}
when $\alpha\in\N^n$ and $p\in\N$. Thus, the radius of convergence of the Taylor series expansion  of $u(\cdot, t)$ around points in $\Rn$ is $+\infty$ at all times $t>0$. The same holds when $(-\Delta)^m$ is replaced above by other elliptic operators or systems of order $2m$ with constant coefficients. These estimates follow from upper bounds of the holomorphic extension to $\C^n$ of the fundamental solutions of higher order parabolic equations or systems with constant coefficients \cite[p.15 (15);  p.47-48 Theorem 1.1 (3)]{Eidelman} and the fact that a function $f$ in $C^\infty(\Rn)$ verifies
\begin{equation*}
|\partial_x^\alpha f(0)|\le |\alpha|!^{\frac 1{2m}}\rho^{-1-|\alpha|},\ \text{for all}\ \alpha\in\N^n,\ \text{for some}\  0<\rho\le 1,
\end{equation*}
if and only if $f$ is a holomorphic in $\C^n$ and
\begin{equation*}
|f(z)|\le e^{N |z|^{\frac{2m}{2m-1}}},\ \text{for all}\ z\in\C^n\ \text{and for some}\ N\ge 1.
\end{equation*}

To prove the observability inequalities in Theorems \ref{6131} through \ref{bangbang-onecontrol} we apply Lemmas \ref{propagation} or  \ref{higherderivative} to $u(t)$ over $\Omega$ and to $u(x,\cdot)$ over roughly $(\frac t2,t)$ for $x$ in $\Omega$ and $0<t\le T$, with $u$ a solution to one of the above systems. To get the result of these applications  compatible with the {\it telescoping series} method - to make sure that a certain telescoping series converges - we need better quantifications of the space-time analyticity of the solutions to \eqref{E: 6081}, \eqref{6084}, \eqref{anas} and \eqref{heats2} than the ones in \eqref{E: kimalito} or within the available literature \cite{Friedman2,Tanabe, Friedman1,Eidelman,KinderlehrerNirenberg,Komatsu}, where the Taylor series expansion of $u(\cdot, t)$ around a point $x_0$ in $\overline\Omega$ is known to converge absolutely only at points whose distance from $x_0$ is less than a fixed constant multiple of $\root{2m}\of t$.

For our purpose, we need to find a quantification of the space-time analyticity which implies that the space-time Taylor series expansion of solutions converge absolutely over $B_\rho(x)\times ((1-\rho)t,(1+\rho)t)$, for some $0<\rho\le 1$, when $(x,t)$ is in $\overline\Omega\times (0,1]$. Thus, independently of $0<t\le 1$ in the space variable.

E. M. Landis and O. A. Oleinik developed in \cite{LandisOleinik} a reasoning which reduces the study of the strong unique continuation property within characteristic hyperplanes for solutions to {\it time independent} parabolic evolutions to its elliptic counterpart. They informed their readers \cite[p. 190]{LandisOleinik} that their argument implies the space-analyticity of solutions to {\it time-independent} linear parabolic equations from its corresponding elliptic counterpart though they did not bother to quantify their claim. Here, we quantify each step of their reasonings and get the following quantitative estimate.

\begin{lemma}\label{L: 6086}	
There is $\rho=\rho(m,n,\delta)$, $0<\rho\le1$, such that
\begin{equation*}
|\partial_x^\alpha\partial_t^p u(x,t)| \le e^{1/\rho t^{1/(2m-1)}} \rho^{-|\alpha |-p}\, |\alpha|!\, p!\,t^{-p} \|u_0\|
_{L^2(\Omega)},
\end{equation*}
when $x\in \overline\Omega$, $0<t\le 1$, $\alpha\in \N^n$, $p\ge 0$, $2m$ is the order of the evolution and $u$ verifies \eqref{E: 6081}, \eqref{6084}, \eqref{anas} or \eqref{heats2}.
\end{lemma}

It provides a better bound than \eqref{E: kimalito} in \cite{Friedman2,Tanabe, Friedman1,Eidelman,KinderlehrerNirenberg,Komatsu} and it is good, as described above, for our applications to Control Theory. Also observe that Lemma \ref{L: 6086} is somehow in between \eqref{E: kimalito} and \eqref{E: quecachondeo}, since
\begin{equation*}
\sup_{t>0}\, {t^{-\frac{|\alpha|}{2m}}} e^{1/\rho t^{1/(2m-1)}}\lesssim |\alpha|!^{1-\frac{1}{2m}},\ \text{for}\ \alpha\in\N^n.
\end{equation*}

Lemma \ref{L: 6086} also holds for solutions to {\it time independent} linear parabolic equations associated to elliptic and possibly non self-adjoint equations of order $2m$ with analytic coefficients. We do not complete the details here. The readers can obtain such quantitative estimates from \cite{LandisOleinik} and with arguments similar to those in Section \ref{S:2}.

We believe that Lemma \ref{L: 6086} holds when the coefficients of the parabolic evolution are time dependent and verify  \eqref{E condiciongeneral} but so far we do not know how to prove it.

The paper is organized as follows: Section \ref{S:2} proves Lemma \ref{L: 6086}; Section \ref{S:3} shows the results related to higher order parabolic equatons, Section \ref{S:4} verifies the ones for systems and Section \ref{S:5} recalls some applications of Theorems \ref{6131}, \ref{6106}, \ref{nball} and \ref{bangbang-onecontrol}  to Control Theory. One can find analogous  applications of Theorems \ref{614g3}, \ref{6090}, \ref{614g2} and \ref{614g1}.

\section{Proof of Lemma \ref{L: 6086}}\label{S:2}

We first prove Lemma \ref{L: 6086} for solutions to \eqref{E: 6081}. Other time-independent parabolic evolutions associated to {\it self-adjoint} elliptic  scalar operators or systems with analytic coefficients are treated similarly.
\begin{proof}[Proof of Lemma \ref{L: 6086} for \eqref{E: 6081}]
 Let $\{e_j\}_{j\geq1}$ and $\{w_j^{2m}\}_{j\geq1}$ be respectively the sets of $L^2(\Omega)$-normalized eigenfunctions  and eigenvalues for
$(-\Delta)^m$ with zero lateral Dirichlet boundary conditions; i.e.,
\begin{equation*}
\begin{cases}
(-1)^m\Delta^me_j-w_j^{2m}e_j=0,\ &\text{in}\ \Omega,\\
e_j=\nabla e_j=\cdots=\nabla^{m-1}e_j=0,\ &\text{on}\ \p\Omega.
\end{cases}
\end{equation*}
Take $u_0=\sum_{j\geq1}a_je_j$, with $\sum_{j\geq1}a_j^2<+\infty$ and define
\begin{equation*}
u(x,y,t)=\sum_{j\geq 1}a_je^{-tw_j^{2m}}e_j(x)X_j(y),\ \text{for}\ x\in
\overline{\Omega},\ y\in\mathbb{R}\ \text{and}\ t>0,
\end{equation*}
with
\begin{equation}\label{614c2}
X_j(y)=
\begin{cases}
e^{w_jy},\  &\text{when $m$ is odd},\\
e^{w_jye^{\frac{\pi i}{2m}}},\  &\text{when $m$ is even},\\
\end{cases}
\end{equation}
where $i=\sqrt{-1}$. Then, $u(x,t)=u(x,0,t)$, solves \eqref{E: 6081} with initial datum $u_0$ and
\begin{equation}\label{611}
 \p_t^p u(x,y,t)=\sum_{j\geq 1}\left(-1\right)^p a_j\, w_j^{2mp} e^{-tw_j^{2m}}e_j(x)X_j(y),\;\;x\in\overline{\Omega},\
 y\in\mathbb{R}.
\end{equation}
Moreover,
\begin{equation*}\label{614c1}
\begin{cases}
(\p_y^{2m}+\Delta_x^{m})(\p_t^p u(\cdot,\cdot,t))=0,\ &\text{in}\ \Omega\times\mathbb{R},\\
\p_t^p u(\cdot,\cdot,t)=\nabla(\p_t^p u(\cdot,\cdot,t))
=\cdots=\nabla^{m-1}(\p_t^p u(\cdot,\cdot,t))=0,\ &\text{on}\
\p\Omega\times\mathbb{R}.
\end{cases}
\end{equation*}
Because $\p\Omega$ is analytic, the quantitative estimates on the analyticity up to the boundary for solutions to elliptic equations with analytic coefficients and null-Dirichlet data over nearby  analytic boundaries  (See \cite[Ch. 5]{Morrey} or \cite[Ch. 3]{FJohn2}), show that there is $\rho=\rho(\Omega)$, $0<\rho\le 1$, such that for $x_0$ in $\overline{\Omega}$ and $0<R\le 1$
\begin{multline}\label{614c3}
\|\p_x^\alpha\p_t^p u(\cdot,\cdot, t)\|_{L^\infty(B_{R/2}(x_0,0)\cap\Omega\times\mathbb{R})}\\
\leq |\alpha|!\, \rho^{-1-|\alpha|} R^{-|\alpha|}
\left(\text{\rlap|{$\int_{B_{R}(x_0,0)\cap\Omega\times\R}$}}
|\partial^p_t u(x,y,t)|^2\,dxdy\right)^{\frac 12}.
\end{multline}
Because
\begin{equation}\label{614c4}
\int_{B_{R}(x_0,0)\cap\Omega\times\mathbb{R}}|\p_t^p u(x,y,t)|^2\,dxdy
\leq
\int_{-R}^{R}\int_{\Omega}|\p_t^p u(x,y,t)|^2\,dxdy,
\end{equation}
we have from \eqref{614c2}, \eqref{611} and the orthogonality of $\{e_j\}_{j\geq1}$ in $L^2(\Omega)$ that
\begin{equation*}
\begin{split}
&\int_{\Omega}|\p_t^p u(x,y,t)|^2\,dx=\int_\Omega
\Big|\sum_{j\geq 1}\left(-1\right)^p a_j\, w_j^{2mp} e^{-tw_j^{2m}}e_j(x)X_j(y)\Big|^2\,dx\\
&=\sum_{j\geq 1}a_j^2\, w_j^{4mp} e^{-2tw_j^{2m}} |X_j(y)|^2\leq \sum_{j\geq 1}a_j^2\, w_j^{4mp} e^{-2tw_j^{2m}} e^{2w_j|y|}\\
&\leq \max_{j\geq1}\big\{w_j^{4mp} e^{-tw_j^{2m}}\big\}
\max_{j\geq1}\big\{e^{-tw_j^{2m}+2w_j|y|}\big\}\sum_{j\geq1}a_j^2.
\end{split}
\end{equation*}
Next, from Stirling's formula
\begin{equation*}
\max_{x\ge0}x^{2p}e^{-xt}=t^{-2p}\left(2p\right)^{2p}
e^{-2p}\lesssim \left(\frac{2}{t}\right)^{2p}p!^2
,\;\;\mbox{when}\;\; t>0\ \text{and}\ p\ge 0,
\end{equation*}
and the fact that
\begin{equation*}
\max_{x\ge 0}e^{-tx^{2m}+2x|y|}= e^{\left(2-\frac 1m\right)\left(\frac{|y|}{mt}\right)^{\frac1{2m-1}}},\ \text{when}\ t>0,\ m\ge 1,
\end{equation*}
we get that
\begin{equation*}
\int_{\Omega}|\p_t^p u(x,y,t)|^2\,dx
\lesssim\left(\frac{2}{t}\right)^{2p}p !^2 e^{2|y|\left(
\frac{|y|}{mt}\right)^{\frac{1}{2m-1}}}
\sum_{j\geq1}a_j^2.
\end{equation*}
This, along with \eqref{614c3}, \eqref{614c4} and the choice of $R=1$ show that
\begin{equation*}
\|\p_x^\alpha\p_t^p u(\cdot,\cdot, t)\|_{L^\infty(B_{1/2}(x_0,0)\cap\Omega\times\mathbb{R})}\le N|\alpha|!\,p !\,\rho^{-|\alpha|}
\left(\frac{2}{t}\right)^{p}\, e^{Nt^{-\frac{1}{2m-1}}}
\left(\sum_{j\geq1}a_j^2\right)^{1/2}.
\end{equation*}
In particular,
\begin{equation*}\label{6069}
|\p_x^\alpha\p_t^p u(x, t)|
\leq e^{1/\rho t^{1/(2m-1)}}\rho^{-|\alpha|-p}|\alpha|!\, p !\, t^{-p}\|u_0\|_{L^2(\Omega)}.
\end{equation*}
\end{proof}
\begin{remark}\label{R:8}
The last proof extends to the case $m\ge 2$ its analog  for $m=1$ in \cite[Lemma 6]{AEWZ}. There the authors used that the Green's function over $\Omega$ for $\Delta-\partial_t$ with zero lateral Dirichlet conditions has Gaussian upper bounds. The later shows that one can derive  \cite[Lemma 6]{AEWZ} without knowledge of upper bounds for the Green's function with lateral Dirichlet conditions of the parabolic evolution.
\end{remark}
We now give a proof of Lemma \ref{L: 6086} for solutions to the systems \eqref{anas} and \eqref {heats2}. Other time-independent parabolic evolutions associated to possibly {\it non self-adjoint} elliptic  scalar equations with analytic coefficients over $\overline\Omega$ are treated similarly.
\begin{proof}[Proof of Lemma \ref{L: 6086} for \eqref{anas}]
The proof of Lemma 3 requires first global bounds on the time-analyticity of the solutions, Lemma \ref{L: 60100} below. Of course, there is  plenty of literature on the time-analyticity of solutions to abstract evolutions \cite{KatoTanabe,HKomatsu,Massey,Tanabe2} but we give here a proof of Lemma \ref{L: 60100} because it serves better our purpose.
\begin{lemma}\label{L: 60100}	
There is $\rho=\rho(\delta)$, $0<\rho\le1$, such that
\begin{equation*}
t^p\|\partial_t^p \mathbf u(t)\|_{L^2(\Omega)}+t^{p+\frac 12}\|\nabla\partial_t^p \mathbf u(t)\|_{L^2(\Omega)}\le \rho^{-1-p} p!\,\|\mathbf u_0\|_{L^2(\Omega)},
\end{equation*}
when $p\ge 0$, $0< t\le 2$ and $\mathbf u$ verifies \eqref{anas} or \eqref{heats2}.
\end{lemma}
\begin{proof}[Proof of Lemma \ref{L: 60100}] Let $\mathbf u$ solve \eqref{anas}. When $\mathbf u_0$ is in $C^\infty_0(\Omega)$, the solution $\mathbf u$ to \eqref{anas} is in $C^\infty(\overline\Omega\times [0,+\infty))$ \cite{Friedman1}. By the local energy inequality for \eqref{anas} there is $\rho=\rho(\delta)>0$ such that
\begin{equation*}
\sup_{0\le t\le 2}\|\mathbf u(t)\|_{L^2(\Omega)}\le \rho^{-1}\|\mathbf u_0\|_{L^2(\Omega)}.
\end{equation*}
Multiply first the equation satisfied by $\partial_t^p\mathbf u$,
\begin{equation}\label{E: unauecial}
\begin{cases}
\partial_t^{p+1}\mathbf u-\mathbf L\partial_t^p \mathbf u=0,\ &\text{in}\ \Omega\times (0,+\infty),\\
\partial_t^p\mathbf u=0,\ &\text{in}\ \partial\Omega\times (0,+\infty),
\end{cases}
\end{equation}
by $t^{2p+2}\partial_t^{p+1}\mathbf u$, after by $t^{2p+1}\partial_t^p\mathbf u$ and integrate by parts over $\Omega_T=\Omega\times (0,T)$, $0<T\le 2$, the two resulting identities. These, standard energy methods, H\"older's inequality together with \eqref{E: condicion3} \eqref{E: condicon5} and \eqref{E: condici�n4} imply that
\begin{multline}\label{E: primeraacotacion}
T^{p+1}\|\nabla\partial_t^p\mathbf u(T)\|_{L^2(\Omega)}+\|t^{p+1}\partial_t^{p+1}\mathbf u\|_{L^2(\Omega_T)}\\\lesssim \|t^{p}\partial_t^{p}\mathbf u\|_{L^2(\Omega_T)}
+\left(p+1\right)^\frac 12\|t^{p+\frac 12}\partial_t^{p}\nabla\mathbf u\|_{L^2(\Omega_T)},
\end{multline}
\begin{equation}\label{E: segundaacotacion}
T^{p+\frac 12}\|\partial_t^p\mathbf u(T)\|_{L^2(\Omega)}+\|t^{p+\frac 12}\partial_t^{p}\nabla\mathbf u\|_{L^2(\Omega_T)}\lesssim \left(p+1\right)^\frac 12\|t^{p}\partial_t^{p}\mathbf u\|_{L^2(\Omega_T)}.
\end{equation}
Thus,
\begin{equation}\label{E: iteracion}
\|t^{p+1}\partial_t^{p+1}\mathbf u\|_{L^2(\Omega_T)}\le \rho^{-1}\left(p+1\right)\|t^{p}\partial_t^{p}\mathbf u\|_{L^2(\Omega_T)},\ \text{for}\ p\ge 0
\end{equation}
and the iteration of \eqref{E: iteracion} and the local energy inequality show that
\begin{equation*}
\|t^p\partial_t^p \mathbf u(t)\|_{L^2(\Omega_T)}\le \rho^{-1-p} p!\,\sqrt T\,\|\mathbf u_0\|_{L^2(\Rn)},\ \text{for}\ p\ge 0.
\end{equation*}
This combined with \eqref{E: primeraacotacion} and \eqref{E: segundaacotacion} implies Lemma \ref{L: 60100}.

\end{proof}
The next step is to show that we can realize $\mathbf u(x,t)$ and all its partial derivatives with respect to time as  functions with one more space variable, say $x_{n+1}$, which satisfy in the $(X,t)=(x,x_{n+1}, t)$ coordinates a {\it time-independent} parabolic evolution associated to a {\it self-adjoint} elliptic system with analytic coefficients over $\Omega\times (-1,1)\times (0,+\infty)$ and with zero boundary values over $\partial\Omega\times(-1,1)\times (0,+\infty)$. To accomplish it, consider the system $\mathbf S=\left(S^1,\dots,S^\ell\right)$, which  acts on functions $\mathbf w$ in $C^\infty(\Omega\times\R,\R^\ell)$, $\mathbf w=\left(w^1,\dots,w^\ell\right)$, as
\begin{multline*}
S^{\xi}\mathbf w=\sum_{i,j=1}^{n+1}\sum_{\eta=1}^{\ell}\partial_{x_i}\left(\tilde a_{ij}^{\xi\eta}(X)\partial_{x_j}w^\eta\right)\\+\sum_{\eta=1}^\ell
\left[\partial_{x_{n+1}}\left(x_{n+1}c^{\xi\eta}(x)w^{\eta}\right)-
x_{n+1}c^{\eta\xi}(x)\partial_{x_{n+1}}w^\eta\right],
\end{multline*}
for $\xi =1,\dots, \ell$, where for $\xi, \eta= 1,\dots,\ell$,
\begin{equation*}
\tilde a_{ij}^{\xi\eta}(X)=
\begin{cases}
a_{ij}^{\xi\eta}(x),\ &\text{for}\  i,j=1,\dots,n,\\
x_{n+1}b^{\xi\eta}_j(x),\ &\text{for}\ i=n+1,\ j=1,\dots,n\\
x_{n+1}b^{\eta\xi}_i(x),\ &\text{for}\ i=1,\dots,n,\ j=n+1,\\
M\delta_{\xi\eta},\ &\text{for}\ i=j=n+1.
\end{cases}
\end{equation*}
Set $Q_R=\Omega\times\left(-R,R\right)$ and $\partial_lQ_R=\partial\Omega\times\left(-R,R\right)$, the \lq\lq lateral\rq\rq boundary of $Q_R$. From \eqref{E: condicon5}, $\mathbf S$ is a self-adjoint system and for large $M=M(\delta)$, the matrices of coefficients $\tilde a_{ij}^{\xi\eta}$ verify one the ellipticity conditions \eqref{E: condici�n4} or \eqref{E:Hadamard2} with $\delta$ replaced by $\frac \delta 2$ over $Q_{1}$ when the original coefficients $a^{\xi\eta}_{ij}$ verify respectively \eqref{E: condici�n4} or \eqref{E:Hadamard2}. Choosing $M$ larger if it is necessary, we may assume that
\begin{equation}\label{E: desigualdad coercividad}
\tfrac\delta 2\|\nabla_X\boldsymbol\varphi\|_{L^2(Q_{1})}^2\le -\int_{Q_{1}}\mathbf S\boldsymbol\varphi\cdot\boldsymbol\varphi\,dX\le \tfrac 2\delta\|\nabla_X\boldsymbol\varphi\|_{L^2(Q_{1})}^2,
\end{equation}
when $\boldsymbol\varphi$ is in $W^{1,2}_0(Q_{1})$ and $\nabla_X= \left(\nabla_x,\partial_{x_{n+1}}\right)$.  Also, $\mathbf S\boldsymbol\varphi (X)=\mathbf L\mathbf v(x)$, when $\boldsymbol\varphi(X)=\mathbf{v}(x)$ and for $\mathbf w(X,t)=\partial_t^p\mathbf u(x,t)$, $p\ge 0$, we have
\begin{equation*}
\begin{cases}
\partial_t\mathbf w-\mathbf S\mathbf w=0,\ &\text{in}\ Q_{1}\times (0,+\infty),\\
\mathbf w=0,\ &\text{in}\ \partial_lQ_{1}\times (0,+\infty).
\end{cases}
\end{equation*}
The symmetry, coerciveness and compactness of the operator mapping functions $\mathbf f$ in $L^2(Q_{1})^m$ into the unique solution $\boldsymbol{\varphi}$ in $W^{1,2}_0(Q_{1})^m$ to
\begin{equation*}
\begin{cases}
\mathbf{S}\boldsymbol\varphi=\mathbf f,\ &\text{in}\ Q_{1},\\
\boldsymbol\varphi=0,\ &\text{in}\ \partial Q_{1}
\end{cases}
\end{equation*}
\cite[Prop. 2.1]{Giaquinta} gives the existence of a complete orthogonal system $\{\mathbf e_k\}$ in $L^2(Q_{1})^m$ of eigenfunctions, $\mathbf e_k=(e_k^1,\dots,e_k^m)$, satisfying
\begin{equation*}
\begin{cases}
\mathbf S\mathbf{e}_k+\omega_k^2\,\mathbf e_k=0,\ &\text{in}\ Q_{1},\\
\mathbf e_k=0,\ &\text{in}\ \partial Q_{1},
\end{cases}
\end{equation*}
with eigenvalues $0<\omega_1^2\le\dots\omega_k^2\le\dots$. Fix $0<T\le 1$ and for $(X,t)$ in $Q_{1}\times (\frac T2,+\infty)$ consider
\begin{equation*}
\mathbf w_1(X,t)=\sum_{j\geq1}a_je^{-w^2_j(t-T/2)}\mathbf e_j(X),
\end{equation*}
with
\begin{equation}\label{def:a}
a_j=\int_{Q_{1}}\partial_t^p\mathbf u(x,\tfrac T2)\mathbf e_j(X)\,dX.
\end{equation}
Clearly, $\mathbf w_1(X,\tfrac T2)=\partial_t^p\mathbf u(x,\tfrac T2)$ in $Q_{1}$ and by the multiplications of the equation verified by $\mathbf w_1$, first by $\mathbf w_1$, after by $\partial_t\mathbf w_1$ and the integration by parts of the resulting identities over $Q_{1}\times (\frac T2,\tau)$, for $\frac T2<\tau\le 2T$, we get
\begin{multline*}
\|\mathbf w_1\|_{L^\infty(\frac T2, 2T; L^2(Q_{1}))}+\sqrt T\,\|\nabla_X\mathbf w_1\|_{L^\infty(\frac T2, 2T; L^2(Q_{1}))}\\
\lesssim \|\mathbf \partial_t^p\mathbf u(\tfrac T2)\|_{L^2(\Omega)}+ \sqrt T\,\|\nabla\mathbf \partial_t^p\mathbf u(\tfrac T2)\|_{L^2(\Omega)}.
\end{multline*}
From Lemma \ref{L: 60100}
\begin{equation}\label{E: estavade necesaia}
\|\mathbf w_1\|_{L^\infty(\frac T2, 2T; L^2(Q_{1}))}+\sqrt T\,\|\nabla_X\mathbf w_1\|_{L^\infty(\frac T2, 2T; L^2(Q_{1}))}\le \sqrt T\,H(p,T,\rho),
\end{equation}
with
\begin{equation}\label{E: una definici—n}
H(p,T,\rho)=\rho^{-1-p}p!\,T^{-p-\frac12}\|\mathbf u_0\|_{L^2(\Omega)},\ 0<\rho\le 1,\ \rho=\rho(\delta).
\end{equation}

Let $\mathbf w_2$ be the solution to
\begin{equation*}
\begin{cases}
\partial_t\mathbf w_2-\mathbf S\mathbf w_2=0,\ &\text{in}\ Q_1\times (\frac T2,+\infty),\\
\mathbf w_2 =\eta(t)\left(\partial_t^p\mathbf u-\mathbf w_1\right),\ &\text{on}\ \partial Q_1\times (\frac T2,+\infty),\\
\mathbf w_2(0)= \mathbf 0,\ &\text{in}\ Q_1,
\end{cases}
\end{equation*}
where $0\leq\eta\leq1$ verifies $\eta=1$, for $-\infty <t\le T$,
$\eta=0$, for $\frac{3T}2\le t<+\infty$ and $|\partial_t\eta|\le\frac1T$. Observe that because $\partial_t^p\mathbf u=0$ on $\partial\Omega\times (0,+\infty)$, $\partial_lQ_1\subset \partial Q_{1}$ and $\mathbf w_1=0$ on $\partial Q_{1}$, then $\mathbf w_2=0$ on $\partial_lQ_{1}$.

The auxiliary function, $\mathbf v=\mathbf w_2-\eta(t)(\partial_t^p\mathbf u-\mathbf w_1)$ satisfies
\begin{equation*}
\begin{cases}
\partial_t\mathbf v-\mathbf S\mathbf v=-(\partial_t^p\mathbf u-\mathbf w_1)\partial_t\eta\;\;&\text{in}\ Q_1\times(T/2,+\infty),\\
\mathbf v=0\;\;&\text{on}\;\;\partial Q_1\times(T/2,+\infty),\\
\mathbf v(T/2)=0\;\;&\text{in}\ Q_1
\end{cases}
\end{equation*}
and clearly $v\equiv0$ in $Q_1\times[\frac T2,T]$. In particular,
\begin{equation}\label{E: idntidadmaravilosa}
\partial_t^p\mathbf u(x,T)=\mathbf w_1(X,T)+\mathbf w_2(X,T),\ \text{for $X$ in}\ Q_1.
\end{equation}
By the parabolic regularity
\begin{equation*}
\|\mathbf v\|_{L^\infty(T/2,2T;L^2(Q_1))}+\|\nabla_X \mathbf v\|_{L^\infty(T/2,2T;L^2(Q_1))}\lesssim\|(\partial_t^p\mathbf u-\mathbf w_1)
\partial_t\eta\|_{L^2(\frac T2,2T;L^2(Q_1))}
\end{equation*}
and from Lemma \ref{L: 60100} and \eqref{E: estavade necesaia}
\begin{equation*}
\|\mathbf v\|_{L^\infty(T/2,2T;L^2(Q_1))}+\|\nabla_X \mathbf v\|_{L^\infty(T/2,2T;L^2(Q_1))}\lesssim H(p,T,\rho).
\end{equation*}
Because $\mathbf w_2=\mathbf v+\eta(t)\left(\partial_t^p\mathbf u-\mathbf w_1\right)$, we get from the latter, Lemma \ref{L: 60100} and \eqref{E: estavade necesaia}
\begin{equation}\label{E: unaacotacionmas}
\|\mathbf w_2\|_{L^\infty(\frac T2, 2T; L^2(Q_1))}+\|\nabla_X\mathbf w_2\|_{L^\infty(\frac T2, 2T; L^2(Q_1))}\lesssim H(p,T,\rho).
\end{equation}
By separation of variables,
\begin{equation*}
\mathbf w_2(X,t)=\sum_{j=1}^{+\infty}c_je^{-\omega_j^2\left(t-2T\right)}\mathbf e_j(X),\ \text{with}\  c_j=\int_{Q_1}\mathbf w_2(X,2T)\mathbf e_j(X)\,dX,
\end{equation*}
for $t\ge 2T$. From \eqref{E: desigualdad coercividad}, $\omega_1^2\ge \frac\delta 2$ and
\begin{equation}\label{E: decay}
\|\mathbf w_2(t)\|_{L^2(Q_1)}\le  e^{-\frac\delta 2\left(t-2T\right)}\|\mathbf w_2(2T)\|_{L^2(Q_1)},\ \text{when}\ t\ge 2T.
\end{equation}
Also,
\begin{equation*}
-\int_{Q_1}\mathbf S\mathbf w_2(t)\cdot\mathbf w_2(t)\,dX=-\int_{Q_1}\partial_t\mathbf w_2(t)\cdot\mathbf w_2(t)\, dX=\sum_{j=1}^{+\infty}c_j^2\omega_j^2e^{-2\omega_j^2\left(t-2T\right)},
\end{equation*}
for $t\ge 2T$ and the last identity and \eqref{E: desigualdad coercividad} imply that
\begin{equation*}
\|\nabla_X\mathbf w_2(t)\|_{L^2(Q_1)}\le  e^{-\frac \delta 2\left(t-2T\right)}\|\nabla_X\mathbf w_2(2T)\|_{L^2(Q_1)},\ \text{when}\ t\ge 2T.
\end{equation*}
From \eqref{E: unaacotacionmas},   \eqref{E: decay} and the last inequality
\begin{equation}\label{E: otra acotacionqueva}
\|\mathbf w_2(t)\|_{L^2(Q_1)}+  \|\nabla_X\mathbf w_2(t)\|_{L^2(Q_1)}\lesssim e^{-\frac\delta 2\left(t-2T\right)^+} H(p,T,\rho)
\end{equation}
and we may extend $\mathbf w_2$ as zero for $t\le\frac T2$. Set
\begin{equation*}
\widehat{\mathbf w}_2(X,\mu)=\frac 1{\sqrt{2\pi}}\int_{\frac T2}^{+\infty}e^{-i\mu t}\mathbf w_2(X,t)\,dt = \frac 1{\sqrt{2\pi}}\int_{-\infty}^{+\infty}e^{-i\mu t}\mathbf w_2(X,t)\,dt,
\end{equation*}
for $X$ in $Q_1$ and $\mu$ in $\R$. From \eqref{E: otra acotacionqueva}
\begin{equation}\label{E: acotacion temeraria}
\|\widehat{\mathbf w}_2(\mu)\|_{L^2(Q_1)}+\|\nabla_X\widehat{\mathbf w}_2(\mu)\|_{L^2(Q_1)}\lesssim H(p,T,\rho),\ \text{for all}\ \mu\in\R.
\end{equation}
Moreover,
\begin{equation*}
\begin{cases}
\mathbf S\widehat{\mathbf w}_2(X,\mu)-i\mu\widehat{\mathbf w}_2(X,\mu)=0,\ &\text{in}\ Q_1,\\
\widehat{\mathbf w}_2(X,\mu)=0,\ &\text{in}\ \partial_lQ_1,
\end{cases}
\quad\text{for each}\ \mu\in\R.
\end{equation*}
For $\mu\neq 0$, define
\begin{equation}\label{E: unaformulaincreible}
\mathbf v_2(X,\zeta,\mu)=e^{i\zeta\sqrt{|\mu|}}\widehat{\mathbf w}_2(X,\mu),\;\;\zeta\in\mathbb R.
\end{equation}
 Then,
\begin{equation*}
\begin{cases}
\mathbf S \mathbf v_2(X,\zeta,\mu)+i\sgn{(\mu)}\,\partial^2_\zeta\mathbf v_2(X,\zeta,\mu)=0,\ &\text{in}\ Q_1\times\R,\\
\mathbf v_2(X,\zeta,\mu)=0,\ &\text{in}\ \partial_lQ_1\times\R.
\end{cases}
\end{equation*}
As for the equation verified by $\mathbf v_2$, it is elliptic with complex coefficients and its coefficients are independent of the $\zeta$-variable. These and the fact that $\partial^k_{\zeta}\mathbf v_2=0$ on $\partial_lQ_1\times\R$ imply by energy methods \cite{MorreyNirenberg} ($k$ times localized Cacciopoli's inequalities) that
\begin{equation*}\label{E: iteracionclasica}
\|\partial_\zeta^{j+1}\mathbf v_2\|_{L^2(Q_{1-\frac{j+1}{2k}}\times(-1+\frac{j+1}{2k}),1-\frac{j+1}{2k}))}
\le\tfrac k{\rho}\,\|\partial_\zeta^{j}\mathbf v_2\|_{L^2(Q_{1-\frac{j}{2k}}\times(-1+\frac{j}{2k},1-\frac{j}{2k}))},
\end{equation*}
for $j=0,\dots,k-1$, $k\ge 1$, and for some $0<\rho\le 1$, $\rho=\rho(\delta)$. Its iteration gives
\begin{equation*}
\|\partial_\zeta^k\mathbf v_2\|_{L^2(Q_{\frac 12}\times(-\frac 12,\frac 12))}\le k!\,\rho^{-k}\|\mathbf v_2\|_{L^2(Q_{1}\times(-1,1))},\ \text{for}\ k\ge 1,
\end{equation*}
and from \eqref{E: acotacion temeraria} and \eqref{E: unaformulaincreible}
\begin{equation}\label{E: otra cojonudaacota}
\|\partial_\zeta^k\mathbf v_2\|_{L^2(Q_{\frac 12}\times(-\frac 12,\frac 12))}\lesssim k!\,\rho^{-k}H(T,p,\rho),\ \text{for}\ k\ge 1.
\end{equation}
For $\psi$ in $L^2(Q_{\frac 12})$, set $\gamma(\zeta)=\int_{Q_{\frac 12}}\mathbf v_2(X,\zeta,\mu)\overline\psi(X)\,dX$. Then, from \eqref{E: acotacion temeraria}, \eqref{E: unaformulaincreible} and \eqref{E: otra cojonudaacota}
\begin{equation*}
\|\gamma^{(k)}\|_{L^\infty(-\frac 12,\frac 12)}\lesssim \rho^{-k}k!\,H(T,p,\rho)\,\|\psi\|_{L^2(Q_{\frac 12})},\ \text{for}\ k\ge 0.
\end{equation*}
Thus, $\gamma(-\tfrac{i\rho} 2)$ can be calculated via the Taylor series expansion of $\gamma$ around $\zeta=0$ and after adding a geometric series
\begin{equation*}
|\gamma(-\tfrac{i\rho} 2)|=\Big|\int_{Q_{\frac 12}} e^{\rho  \sqrt{|\mu|}/2}\widehat{\mathbf w}_2(X,\mu)\overline\psi(X)\,dX\Big|\lesssim \|\psi\|_{L^2(Q_{\frac 12})}\, H(T,p,\rho).
\end{equation*}
All together,
\begin{equation}\label{E: unagozada}
\|\widehat{\mathbf w}_2(\cdot,\mu)\|_{L^2(Q_{\frac 12})}\lesssim e^{-\rho\sqrt{|\mu|}/2}H(T,p,\rho),\ \text{when}\ \mu\in\R.
\end{equation}
Define then,
\begin{equation*}
\mathbf U_2(X,y)=\frac 1{\sqrt{2\pi}}\int_{\R} e^{i\mu T} \widehat{\mathbf w}_2(X,\mu)\cosh{\left(y\sqrt{-i\mu}\right)}\,d\mu,
\end{equation*}
for $(X,y)$ in $Q_{\frac 12}\times \R$, with $\sqrt{-i\mu}=\sqrt{|\mu|}\, e^{-\frac{i\pi}4\sgn{\mu}}$. From \eqref{E: unagozada},
\begin{equation}\label{E:llegandoalfinal}
\|\mathbf U_2(\cdot,y)\|_{L^2(Q_{\frac 12})}\lesssim H(T,p,\rho),\ \text{for}\ |y|\le\frac \rho4\, .
\end{equation}
Observe that $\mathbf U_2$ is in $C^\infty(\overline Q_{\frac 12}\times [-\frac\rho4,\frac\rho4])$ and that one may derive similar bounds for higher derivatives of $\mathbf U_2$. Also,
\begin{equation}\label{E: laecuacionde U2}
\begin{cases}
\mathbf S\mathbf U_2+\partial^2_y\mathbf U_2=0,\ &\text{in}\ Q_{\frac 12}\times (-\frac\rho4,\frac\rho4),\\
\mathbf U_2=0,\ &\text{in}\ \partial_lQ_{\frac 12}\times (-\frac\rho4,\frac\rho4)
\end{cases}
\end{equation}
and
\begin{equation}\label{E: losiuetedeU2}
\mathbf U_2(X,0)=\frac 1{\sqrt{2\pi}}\int_{\R}e^{i\mu T}\widehat{\mathbf w}_2(X,\mu)\,d\mu =\mathbf w_2(X,T),\ \text{in}\ Q_{\frac12}.
\end{equation}
Next,
\begin{equation*}
\mathbf U_1(X,y)=\sum_{j=1}^{+\infty}e^{-\omega_j^2T/2}a_j\mathbf e_j(X)\cosh{(\omega_jy)},
\end{equation*}
with $a_j$ as in \eqref{def:a} satisfies
\begin{equation}\label{E: loqueverificaU1}
\mathbf U_1(X,0)=\mathbf w_1(X,T), \text{in}\ Q_{1},\quad\quad
\begin{cases}
\mathbf S\mathbf U_1+\partial^2_y\mathbf U_1=0,\ &\text{in}\ Q_{1}\times \R,\\
\mathbf U_1=0,\ &\text{in}\ \partial Q_{1}\times \R,
\end{cases}
\end{equation}
and
\begin{equation}\label{E: otracosadeU1}
\sup_{|y|\le1}\|\mathbf U_1(\cdot,y)\|_{L^2(Q_1)}\lesssim e^{1/T}\|\partial_t^p\mathbf u(\tfrac T2)\|_{L^2(\Omega)} \lesssim e^{1/T}H(T,p,\rho).
\end{equation}
Set then, $\mathbf U=\mathbf U_1+\mathbf U_2$. From \eqref{E: laecuacionde U2}, \eqref{E: losiuetedeU2}, \eqref{E: loqueverificaU1} and \eqref{E: idntidadmaravilosa} we have
\begin{equation*}
\begin{cases}
\mathbf S\mathbf U+\partial^2_y\mathbf U=0,\ &\text{in}\ Q_{\frac12}\times (-\frac\rho4,\frac\rho4),\\
\mathbf U=0,\ &\text{in}\ \partial_lQ_{\frac12}\times (-\frac\rho4,\frac\rho4),\\
U(X,0)=\partial_t^p\mathbf u(x,T),\ &\text{in}\ Q_{\frac 12},
\end{cases}
\end{equation*}
while \eqref{E:llegandoalfinal} and \eqref{E: otracosadeU1} show that
\begin{equation}\label{E: loqueverificaU}
\sup_{|y|\le\frac\rho4}\|\mathbf U(\cdot,y)\|_{L^2(Q_{\frac 12})}\lesssim e^{1/T}H(T,p,\rho),\ \text{with}\ \rho=\rho(\delta),\ 0<\rho\le 1.
\end{equation}
Now, $\mathbf S+\partial_y^2$ is an elliptic system with analytic coefficients. This, \eqref{E: loqueverificaU}, the fact that $\mathbf U(X,y)=0$, for $(X,y)=(x,x_{n+1}, y)$ in $\partial\Omega\times(-\frac12,\frac12)\times(-\frac\rho4,\frac\rho4)$ and that $\partial\Omega$ is analytic imply that there is $\rho=\rho(\delta)$, $0<\rho\le 1$ (See \cite{MorreyNirenberg} or \cite[Ch. II]{Giaquinta}) such that
\begin{equation*}
\|\partial_X^\gamma\partial_y^q\mathbf U(X,y)\|_{L^\infty(Q_{\frac14}\times (-\frac\rho4,\frac\rho4))}\le\rho^{-|\gamma|-q}|\gamma|!\,q!\,e^{1/T}H(T,p,\rho),\ \text{for}\ \gamma\in\N^{n+1},\ q\in\N.
\end{equation*}
Finally, $\mathbf U(X, 0)=\partial_t^p\mathbf u(x,T)$ in $\overline\Omega$ and Lemma \ref{L: 6086} follows from the latter and \eqref{E: una definici—n}.
\end{proof}
\begin{remark}\label{R: unremarkillomas}
Observe that we did not use quantitatively the smoothness of $\partial\Omega$ in the proof of Lemma \ref{L: 60100} and that to get the quantitative estimate of Lemma \ref{L: 6086} over only $B_{\frac\delta 2}(x_0)\cap\overline\Omega\times (0,T]$,  with $x_0$ in $\overline\Omega$ and $\delta$ as in \eqref{E:descripcionfrontera}, it suffices to know that either $B_\delta(x_0)\subset\Omega$ or that $\partial\Omega\cap B_{\delta}(x_0)$ is real-analytic.
 
\end{remark}
\section{Observability for higher order elliptic equations}\label{S:3}
We can now explain the proof of Theorem \ref{6131} by making use of Lemmas \ref{propagation} and \ref{L: 6086}
.
\begin{proof}[Proof of Theorem \ref{6131}]
From Lemma \ref{L: 6086}
\begin{equation*}
\begin{split}
|\partial_x^\alpha u(x,L)|
&\leq e^{1/\rho L^{1/(2m-1)}}|\alpha|!\,\rho^{-|\alpha|} \|u(0)\|_{L^2(\Omega)},\;\;\text{for}\
x\in\overline{\Omega}\ \text{and}\ 0<L\le T
\end{split}
\end{equation*}
and from Lemma \ref{propagation} there are $N=N(\Omega,|\omega|,\rho)$ and  $\theta=\theta(\Omega,|\omega|,\rho)$ in $(0,1)$ such that
\begin{equation}\label{923c1}
\|u(L)\|_{L^2(\Omega)}\leq N\|u(L)\|_{L^1(\omega)}
^{\theta}M^{1-\theta},\;\;\;\text{with}\;\;M=Ne^{N/L}
\|u(0)\|_{L^2(\Omega)},
 \end{equation}
when $\omega\subset\Omega$ is a measurable set with a positive measure. Set for each $t\in(0,T)$,
\begin{equation*}
\mathcal{D}_t=\{x\in\Omega: (x,t)\in\mathcal{D}\}\quad \text{and}\quad E=\{t\in(0,T): |\mathcal{D}_t|\geq |\mathcal{D}|/(2T)\}.
\end{equation*}
By Fubini's theorem, $\mathcal{D}_t$ is measurable for a.e. $t\in(0,T)$, $E$ is measurable in $(0,T)$ with $|E|\geq |\mathcal{D}|/(2|\Omega|)$
and $\chi_{E}(t)\chi_{\mathcal{D}_t}(x)\leq \chi_{\mathcal{D}}(x,t)$ over $\Omega\times(0,T)$. Next, let $q\in(0,1)$ be a
constant to be determined later and $l$ be a Lebesgue point of $E$. Then, from \cite[Lemma 2]{AEWZ} there is a monotone decreasing sequence  $\{l_k\}_{k\geq 1}$ satisfying $\lim_{k\to\infty} l_k=l$, $l<l_1\le T$,
\begin{equation}\label{6161can1}
l_{k+1}-l_{k+2}=q(l_k-l_{k+1})
\quad \text{and}\quad
|(l_{k+1},l_k)\cap E|\geq \frac{l_k-l_{k+1}}{3},\;\;k\in\mathbb{N}.
\end{equation}
Define
\begin{equation*}
\tau_k=l_{k+1}+(l_{k}-l_{k+1})/6, \;\;\;k\in\mathbb{N}.
\end{equation*}
From \eqref{923c1} there are $N=N(\Omega,|\mathcal{D}|,T,\rho)$ and $\theta=\theta(\Omega,|\mathcal{D}|,T,\rho)$, $0<\theta <1$, such that
\begin{equation*}
\|u(t)\|_{L^2(\Omega)}\leq \Big(N
e^{\frac{N}{(l_{k}-l_{k+1})^{1/(2m-1)}}}
\|u(t)
\|_{L^1(\mathcal{D}_t)}
\Big)^{\theta}\|u(l_{k+1})\|_{L^2(\Omega)}
^{1-\theta},
\end{equation*}
when $t\in[\tau_k, l_{k}]\cap E$. Integrating the above inequality  over $(\tau_k,l_{k})\cap E$, from Young's inequality and the standard energy estimate for the solutions to \eqref{E: 6081}, we have that for each $\e>0$,
\begin{equation*}
\begin{split}
\|u(l_{k})\|_{L^2(\Omega)}&\leq \e\|u(l_{k+1})\|_{L^2(\Omega)}\\
&\;\;\;\;\;\;+\e^{-\frac{1-\theta}{\theta}}
  Ne^{\frac{N}{(l_{k}-l_{k+1})^{1/(2m-1)}}}
  \int^{l_{k}}_{l_{k+1}}\chi_{E}
  \|u(t)\|_{L^1(\mathcal{D}_t)}\;dt.
\end{split}
\end{equation*}
Multiplying  the above inequality by $\e^{\frac{1-\theta}\theta}
e^{-\frac{N}{(l_{k}-l_{k+1})^{1/(2m-1)}}}$, replacing $\e$ by $\e^\theta$
and finally choosing $\e=e^{-\frac{1}{(l_{k}-l_{k+1})^{1/(2m-1)}}}$ in the resulting  inequality, we obtain that
\begin{equation*}\label{c20}
\begin{split}
  &e^{-\frac{N+1-\theta}{(l_{k}-l_{k+1})^{1/(2m-1)}}}\|u(l_{k})
  \|_{L^2(\Omega)}-  e^{-\frac{N+1}{(l_{k}-l_{k+1})^{1/(2m-1)}}}\|u(l_{k+1})\|
  _{L^2(\Omega)}\\
  &\;\;\;\;\leq N\int^{l_{k}}_{l_{k+1}}\chi_{E}
 \|u(t)
\|_{L^1(\mathcal{D}_t)}dt.
\end{split}
\end{equation*}
Therefore, fixing  $q$ in \eqref{6161can1} as $q=\Big(\frac{N+1-\theta}{N+1}\Big)^{2m-1}$, we have
\begin{equation}\label{10can2}
\begin{split}
&e^{-\frac{N+1-\theta}{(l_{k}-l_{k+1})^{1/(2m-1)}}}\|u(l_{k})
  \|_{L^2(\Omega)}-  e^{-\frac{N+1-\theta}{(l_{k+1}-l_{k+2})^{1/(2m-1)}}}\|u(l_{k+1})\|_
  {L^2(\Omega)}\\
&\leq N\int^{l_{k}}_{l_{k+1}}\chi_{E}
  \|u(t)
\|_{L^1(\mathcal{D}_t)}dt.
\end{split}
\end{equation}
Summing (\ref{10can2}) from $k=1$ to $+\infty$  completes the proof (the telescoping series method).
\end{proof}

To deal with the boundary observability inequalities for the fourth order parabolic evolution, let $\Omega_\rho=\{x\in\mathbb{R}^n : d(x,\overline\Omega)
<\rho\}$, with $\rho>0$ sufficiently small. By the inverse function theorem for analytic functions, $\Omega_\rho$ is a domain with analytic boundary (cf. \cite[p. 249]{AE}) and by standard extension arguments (cf. \cite[Chapter I, Theorem 2.3]{FursikovOImanuvilov}), the interior null controllability of the system
\begin{equation*}
\begin{cases}
\p_t u+\Delta^2 u=\chi_{_{\Omega_\rho\setminus\Omega}}
f,\ &\text{in}\ \Omega_\rho\times(0,T),\\
u=\frac{\partial u}{\partial\nu}=0,\ &\text{on}\ \p\Omega_\rho\times(0,T),\\
u(0)=u_0,\ &\text{in}\ \Omega_\rho,
\end{cases}
\end{equation*}
 with initial datum $u_0$ in $L^2(\Omega)$ is a consequence of Theorem \ref{6131} (See also Remark~\ref{102c1}) by standard duality arguments (HUM method) \cite{Lions1}.
The later implies that there are controls $g_1$ and $g_2$ in $L^2(\partial\Omega\times(0,T))$ with
$$\|g_k\|_{L^2(\partial\Omega\times(0,T))}\leq Ne^{\frac{N}{T^{1/3}}}\|u_0\|_{L^2(\Omega)},\;\;\;k=1,2,$$
such that the solution $u$ to
\begin{equation*}\label{6062}
\begin{cases}
\p_t u+\Delta^2 u=0,\ &\text{in}\ \Omega
\times(0,T),\\
u=g_1,\;\;\frac{\partial u}{\partial\nu}=g_2,\ &\text{on}\ \p\Omega\times(0,T),\\
u(0)=u_0,\ &\text{in}\ \Omega,
\end{cases}
\end{equation*}
verifies $u(T)\equiv 0$.
By a standard duality argument, this full boundary null controllability in turn implies the observability inequality
\begin{equation*}\label{6063}
\|\varphi(0)\|_{L^2(\Omega)}\le e^{N/T^{1/3}}\left[\|\tfrac{\partial\Delta \varphi}{\partial\nu}\|_{L^2(\partial\Omega\times(0,T))}+\|\Delta \varphi\|_{L^2(\partial\Omega\times(0,T))}\right],
\end{equation*}
for solutions $\varphi$ to the dual equation
\begin{equation*}
\begin{cases}
-\p_t \varphi+\Delta^2 \varphi=0,\ &\text{in}\ \Omega\times(0,T),\\
\varphi=\frac{\partial\varphi}{\partial\nu}=0,\ &
\text{on}\ \p\Omega\times(0,T),
\end{cases}
\end{equation*}
with initial datum $\varphi(T)=\varphi_T$ in $L^2(\Omega)$. Thus, we can derive from the above lines and from the decay of the energy the following result.
\begin{lemma}\label{YUANYUAN1}
There is $N=N(\Omega,\delta)$ such that the interpolation inequality
\begin{equation*}\label{YUANYUAN2}
\begin{split}
&\|u(T)\|_{L^2(\Omega)}\\
&\leq \left(e^{N/[\left(\epsilon_2-\epsilon_1\right)T^{\frac 13}]}
\left[\|\tfrac{\partial\Delta u}{\partial\nu}\|_{L^2(\partial\Omega\times [\epsilon_1T,\epsilon_2T]
)}+\|\Delta u\|_{L^2(\partial\Omega\times [\epsilon_1T,\epsilon_2T])}\right]\right)^{\frac 12}
\|u_0\|
_{L^2(\Omega)}^{\frac 12},
\end{split}
\end{equation*}
holds for all solutions $u$ to \eqref{6084} and $0\leq\epsilon_1<\epsilon_2\leq1$.
\end{lemma}
Lemmas \ref{L: 6086} and \ref{YUANYUAN1} imply in a similar way to the reasonings in \cite[Theorem 11]{AEWZ} the following result.
\begin{lemma}\label{6089}
Assume that $E\subset(0,T)$ is a measurable set of positive measure and that $\Gamma_i\subset \p\Omega$, $i=1,2$, are measurable subsets with $|\Gamma_1|,|\Gamma_2|\geq\gamma_0>0$. Then, for each $\eta\in(0,1)$ there are $N=N(\Omega,\eta,\gamma_0,\delta)\geq1$  and $\theta=\theta(\Omega,\eta,\gamma_0,\delta)$, $0<\theta<1$, such that the inequality
\begin{equation}\label{6132}
\begin{split}
&\|u(t_2)\|_{L^2(\Omega)} \le\\
&\left( e^{N/(t_2-t_1)^{1/3}} \int_{t_1}^{t_2}\chi_E(t) \big[\|\tfrac{\partial\Delta u(t)}{\partial\nu}\|_{L^1(\Gamma_1)}+
\|\Delta u(t)\|_{L^1(\Gamma_2)}\big]\,dt \right)^{\theta}\|u(t_1)\|_{L^2(\Omega)}^{1-\theta},
\end{split}
\end{equation}
holds for all solutions $u$ to \eqref{6084},
when $0\le t_1<t_2\le T$ and $|(t_1,t_2)\cap E|\ge \eta (t_2-t_1)$. Moreover,
\begin{equation*}\label{6134}
\begin{split}
&e^{-\frac{N+1-\theta}{(t_2-t_1)^{1/3}}}\|u(t_2)\|_{L^2(\Omega)}- e^{-\frac{N+1-\theta}{(q\left(t_2-t_1\right))^{1/3}}}\|u(t_1)\|_{L^2(\Omega)}\\
&\le N\int_{t_1}^{t_2}\chi_E(t) \big[\|\tfrac{\partial\Delta u(t)}{\partial\nu}\|_{L^1(\Gamma_1)}+
\|\Delta u(t)\|_{L^1(\Gamma_2)}\big]
\,dt,\;\;\mbox{when }\;\;q\ge \Big(\tfrac{N+1-\theta}{N+1}\Big)^3.
\end{split}
\end{equation*}
\end{lemma}
\begin{proof}
Suppose that $0<\eta<1$ satisfies $|(t_1,t_2)\cap E|\geq\eta(t_2-t_1)$. Set
\begin{align*}
 &\tau=t_1+\frac{\eta}{20}(t_2-t_1),\ \tilde{t}_1=t_1+\frac{\eta}{8}(t_2-t_1),\\
&\tilde{t}_2=t_2-\frac{\eta}{8}(t_2-t_1),\ \tilde{\tau}=t_2-\frac{\eta}{20}(t_2-t_1).
\end{align*}
Then, $t_1<\tau<\tilde{t}_1<\tilde{t}_2<\tilde{\tau}<t_2$ and $|E\cap(\tilde{t}_1,\tilde{t}_2)|\geq\tfrac{3\eta}{4}(t_2-t_1)$ and it follows from Lemma \ref{YUANYUAN1} that there is $N=N(\Omega, \eta,\delta)$ such that
\begin{equation*}
\|u(t_2)\|_{L^2(\om)}\leq
e^{N/
\left(t_2-t_1\right)^{1/3}}
\big[\|\tfrac{\partial\Delta u}{\partial\nu}\|
_{L^2(\partial\Omega\times(\tau,\tilde{\tau}))}+
\|\Delta u\|_{L^2(\partial\Omega\times(\tau,\tilde{\tau}))}\big]
^{1/2}
\|u(t_1)\|
_{L^2(\Omega)}^{1/2}.
\end{equation*}
Next,  the inequality
\begin{equation*}
\|\tfrac{\p\Delta u}{\p\nu}\|_{L^2
(\partial\Omega\times(\tau,\tilde{\tau}))}\leq\|\tfrac{\p\Delta u}{\p\nu}\|^{1/2}_{L^1
(\partial\Omega\times(\tau,\tilde{\tau}))}
\|\tfrac{\p\Delta u}{\p\nu}\|^{1/2}_{L^\infty
(\partial\Omega\times(\tau,\tilde{\tau}))}
\end{equation*}
and Lemma \ref{L: 6086} show that
\begin{equation}\label{E: c16}
\|\tfrac{\p\Delta u}{\p\nu}\|_{L^2
(\partial\Omega\times(\tau,\tilde{\tau}))}\leq Ne^{\frac{N}{(t_2-t_1)^{1/3}}}\|u(t_1)\|
_{L^2(\om)}^{1/2}\|\tfrac{\p\Delta u}{\p\nu}\|^{1/2}_{L^1
(\partial\Omega\times(\tau,\tilde{\tau}))}.
\end{equation}
Set $v(x,t)=\tfrac{\p\Delta u}{\p\nu}(x,t)$, for $x$ in $\partial\Omega$ and $t>0$. Then,
\begin{equation}\label{E: c17}
\|v\|_{L^1(\partial\Omega
\times(\tau,\tilde{\tau}))}\leq\left(\tilde{\tau}-\tau\right)
\int_{\partial\Omega}
\|v(x,\cdot)\|_{L^\infty(\tau,\tilde{\tau})}\,d\s\, .
\end{equation}
Denote the interval $[\tau,\tilde{\tau}]$ as $[a,a+L]$, with $a=\tau$ and $L=\tilde{\tau}-\tau=(1-\frac\eta{10})(t_2-t_1)$.
Then,  Lemma \ref{L: 6086} shows that there is $N=N(\om,\eta,\delta)$ such that for each fixed $x$ in $\partial\Omega$, $\tau\le t\le \tilde{\tau}$ and $p\ge 0$,
\begin{equation}\label{E: c18}
|\p_t^p v(x,t)|\leq
\frac{e^{N/\left(t_2-t_1\right)^{1/3}}p!}{\left(\eta
(t_2-t_1)/40\right)^p}
\|u(t_1)\|_{L^2(\om)}\triangleq\frac{Mp!}{\left(2\rho L\right)^\beta}\, ,
\end{equation}
with
\begin{equation*}
M=e^{N/\left(t_2-t_1\right)^{1/3}}\|u(t_1)\|_{L^2(\om)}\ \;\;\text{and}\ \;\;\rho=\frac{\eta}{8\left(10-\eta\right)}\, .
\end{equation*}
Hence it follows from \eqref{E: c18} and \cite[Lemma 13]{AEWZ} that
\begin{equation*}\label{E: c19}
\|v(x,\cdot)\|_{L^{\infty}(\tau,\tilde{\tau})}
\leq \Big(\text{\rlap |{$\int_{E\cap(\tilde{t}_1,\tilde{t}_2)}$}}|v(x,t)|\;dt\Big)
^\gamma
\left(Ne^{N/\left(t_2-t_1\right)^{1/3}}\|u(t_1)\|_{L^2(\om)}
\right)^{1-\gamma},
\end{equation*}
for all $x$ in $\p\Omega$, with $N=N(\om,\eta,\delta)$ and $\gamma=\gamma(\eta)$ in $(0,1)$.
This, along with \eqref{E: c17} and H\"{o}lder's inequality leads to
\begin{equation}\label{E: c20}
\begin{split}
\|v\|&_{L^1(\partial\Omega\times
(\tau,\tilde{\tau}))}\leq e^{\frac{N}{(t_2-t_1)^{1/3}}}
\Big(\int_{E\cap(\tilde{t}_1,\tilde{t}_2)}
\int_{\partial\Omega}
|v(x,t)|\;d\s dt\Big)^\gamma \|u(t_1)\|^{1-\gamma}_{L^2(\om)},
\end{split}
\end{equation}
with some new $N$ and $\gamma$ as above. Because,  $t-t_1\geq \tilde{t}_1-t_1=\tfrac{\eta}{8}\left(t_2-t_1\right)$, when $t\in (\tilde{t}_1,\tilde{t}_2)$, we get from Lemma \ref{L: 6086} that
\begin{equation*}
\|\p_{x'}^\alpha v(t)\|_{L^\infty(\partial\Omega)} \leq
\frac{e^{N/\left(t_2-t_1\right)^{1/3}} |\alpha|!}{\rho^{|\alpha|}}\, \|u(t_1)\|_{L^2(\om)},\ \text{for}\ \alpha\in \mathbb{N}^{n-1}
\end{equation*}
and for some new constants $N=N(\om,\eta,\delta)$ and $\rho=\rho(\Omega,\delta)$.
By the obvious generalization  of Lemma~\ref{propagation} to the case of real-analytic functions defined over analytic  hypersurfaces in $\mathbb{R}^n$, there are $N=N\left(\om, \eta, |\Gamma_1|,\delta\right)$ and $\vartheta=\vartheta\left(\om,|\Gamma_1|,\delta\right)$, $0<\vartheta<1$, such that
\begin{equation}\label{E: c21}
\int_{\partial\Omega}|v(x,t)|\;d\s\leq
\left(\int_{\Gamma_1}|v(x,t)|\;d\s\right)^\vartheta
\left(e^{N/\left(t_2-t_1\right)^{1/3}}\|u(t_1)\|_{L^2(\om)}
\right)^{1-\vartheta},
\end{equation}
when $t\in E\cap(\tilde{t}_1,\tilde{t}_2)$, and it follows from \eqref{E: c20}, (\ref{E: c21}) as well as H\"{o}lder's inequality that
\begin{equation*}\label{E: c22}
\|v\|_{L^1(\p\Omega\times
(\tau,\tilde{\tau}))}\leq\Big(e^{N/\left(t_2-t_1\right)^{1/3}}
\int_{E\cap(\tilde{t}_1,\tilde{t}_2)}\int_{\Gamma_1}
|v(x,t)|\;d\sigma dt\Big)^{\vartheta\gamma}\|u(t_1)\|
_{L^2(\om)}^{1-\vartheta\gamma}.
\end{equation*}
This, together with \eqref{E: c16} and the definition of $v$ leads to
\begin{equation*}
\|\tfrac{\p\Delta u}{\p\nu}\|_{L^2
(\partial\Omega\times(\tau,\tilde{\tau}))}
\leq
\Big(e^{\frac{N}{(t_2-t_1)^{1/3}}}
\int_{E\cap(\tilde{t}_1,\tilde{t}_2)}\int_{\Gamma_1}
|\tfrac{\p\Delta u}{\p\nu}(x,t)|\;d\sigma dt\Big)^{\theta_1}
\|u(t_1)\|
_{L^2(\om)}^{1-\theta_1}.
\end{equation*}
Similarly, we can get that
\begin{equation*}
\|\Delta u\|_{L^2
(\partial\Omega\times(\tau,\tilde{\tau}))}
\leq
\Big(e^{\frac{N}{(t_2-t_1)^{1/3}}}
\int_{E\cap(\tilde{t}_1,\tilde{t}_2)}\int_{\Gamma_2}
\big|\Delta u(x,t)\big|\;d\sigma dt\Big)^{\theta_2}
\|u(t_1)\|
_{L^2(\om)}^{1-\theta_2}.
\end{equation*}
These last two inequalities, as well as the fact that
$$\frac{a^\theta+b^\theta}{2}\leq
\Big(\frac{a+b}{2}\Big)^{\theta},\;\;\text{when}\;\;
a,b>0,\ 0<\theta<1, $$
lead to the first desired  estimate \eqref{6132}. Next, applying Young's inequality to \eqref{6132},
we obtain that for each $\varepsilon>0$,
\begin{equation*}
\begin{split}
\|u(&t_2)\|_{L^2(\Omega)} \le\varepsilon \|u(t_1)\|_{L^2(\Omega)}\\
&+\varepsilon^{-\frac{1-\theta}{\theta}}Ne^{\frac{N}{(t_2-t_1)^{1/3}}}
\int_{t_1}^{t_2}\chi_E(t) \big[\|\tfrac{\partial\Delta u}{\partial\nu}(t)\|_{L^1(\Gamma_1)}+
\|\Delta u(t)\|_{L^1(\Gamma_2)}\big]\,dt.
\end{split}
\end{equation*}
Hence, after some computations, we may get that
\begin{equation*}
\begin{split}
\varepsilon^{1-\theta}&e^{-\frac{N}{(t_2-t_1)^{1/3}}}\|u(t_2)\|_{L^2(\Omega)} -\varepsilon e^{-\frac{N}{(t_2-t_1)^{1/3}}} \|u(t_1)\|_{L^2(\Omega)}\\
&\leq \int_{t_1}^{t_2}\chi_E(t) \big[\|\tfrac{\partial\Delta u}{\partial\nu}(t)\|_{L^1(\Gamma_1)}+
\|\Delta u(t)\|_{L^1(\Gamma_2)}\big]\,dt,\ \text{for all}\ \e>0.
\end{split}
\end{equation*}
Choosing now $\varepsilon=e^{-\frac{1}{(t_2-t_1)^{1/3}}}$ implies the second estimate in the Lemma.
\end{proof}
We now complete the proof of Theorems \ref{6106} and \ref{614g3}.

\begin{proof}[Proof of Theorems \ref{6106} and \ref{614g3}]
Set for each $t\in(0,T)$
\begin{equation*}
\mathcal{J}_t=\{x\in\p\Omega: (x,t)\in\mathcal{J}\}\quad \text{and}\quad E=\{t\in(0,T): |\mathcal{J}_t|\geq |\mathcal{J}|/(2T)\}.
\end{equation*}
By  Fubini's theorem, $\mathcal{J}_t$ is measurable for a.e. $t\in(0,T)$, $E$ is measurable in $(0,T)$ with $|E|\geq |\mathcal{J}|/(2|\p\Omega|)$
and $\chi_{E}(t)\chi_{\mathcal{J}_t}(x)\leq \chi_{\mathcal{J}}(x,t)$ over $\p\Omega\times(0,T)$.
Then, with similar arguments as the ones in the proof of Lemma \ref{6089}, we can get that
 for each $0<\eta<1$,  there are $N=N(\Omega,\eta,|\mathcal{J}|,T,\delta)$  and $\theta=\theta(\Omega,\eta,|\mathcal{J}|,T,\delta)$ with $0<\theta<1$, such that
\begin{equation*}
\begin{split}
&\|u(t_2)\|_{L^2(\Omega)} \le\\
&\left( Ne^{N/(t_2-t_1)^{1/3}} \int_{t_1}^{t_2}\chi_E(t) \big[\|\tfrac{\partial\Delta u}{\partial\nu}(t)\|_{L^1(\mathcal{J}_t)}+
\|\Delta u(t)\|_{L^1(\mathcal{J}_t)}\big]\,dt \right)^{\theta}\|u(t_1)\|_{L^2(\Omega)}^{1-\theta},
\end{split}
\end{equation*}
holds for all solutions $u$ to \eqref{6084},
when $0\le t_1<t_2\le T$ and $|(t_1,t_2)\cap E|\ge \eta (t_2-t_1)$. Moreover,
\begin{equation}\label{610c1}
\begin{split}
&e^{-\frac{N+1-\theta}{(t_2-t_1)^{1/3}}}\|u(t_2)\|_{L^2(\Omega)}- e^{-\frac{N+1-\theta}{(q\left(t_2-t_1\right))^{1/3}}}\|u(t_1)\|_{L^2(\Omega)}\\
&\le N\int_{t_1}^{t_2}\chi_E(t) \big[\|\tfrac{\partial\Delta u}{\partial\nu}(t)\|_{L^1(\mathcal{J}_t)}+
\|\Delta u(t)\|_{L^1(\mathcal{J}_t)}\big]
\,dt,\;\;\mbox{when }\;\;q\ge \Big(\tfrac{N+1-\theta}{N+1}\Big)^3.
\end{split}
\end{equation}

Now, let $\eta=1/3$ and $q=(N+1-\theta)^3/(N+1)^3$ with $N$
and $\theta$ as above.
Assume that $l$ is a Lebesgue point of $E$. By \cite[Lemma 2]{AEWZ}, there is a monotone decreasing sequence  $\{l_k\}_{k\geq 1}$ in $(0,T)$
satisfying $\lim_{k\to\infty} l_k=l$, $l<l_1\le T$ and \eqref{6161can1}.
These, together with \eqref{610c1}, imply that
\begin{equation}\label{610c12}
\begin{split}
&e^{-\frac{N+1-\theta}{(l_k-l_{k+1})^{1/3}}}\|u(l_k)\|
_{L^2(\Omega)}- e^{-\frac{N+1-\theta}{(l_{k+1}-l_{k+2})^{1/3}}}\|u(l_{k+1})\|
_{L^2(\Omega)}\\
&\le N\int_{l_{k+1}}^{l_k}\chi_E(t) \big[\|\tfrac{\partial\Delta u}{\partial\nu}(t)\|_{L^1(\mathcal{J}_t)}+
\|\Delta u(t)\|_{L^1(\mathcal{J}_t)}\big]
\,dt,\ k\in\N.
\end{split}
\end{equation}
Finally, adding up \eqref{610c12} from $k=1$ to $+\infty$ (the telescoping series)
we get that
\begin{equation*}
\begin{split}
\|u(l_1)\|_{L^2(\Omega)}
&\leq Ne^{\frac{N+1-\theta}{(l_1-l_2)^{1/3}}}\int_{l}^{l_1}
\chi_{E}(t)\big[\|\tfrac{\partial\Delta u}{\partial\nu}(t)\|_{L^1(\mathcal{J}_t)}+
\|\Delta u(t)\|_{L^1(\mathcal{J}_t)}\big]
\,dt\\
&\leq N\int_{\mathcal{J}}|\tfrac{\p\Delta u}{\p\nu}(x,t)|+|\Delta u(x,t)|\;d\sigma dt,
\end{split}
\end{equation*}
which completes the proof of Theorem \ref{6106}.

The previous reasonings show that  Lemma~\ref{6089}, as well as \cite[Lemma 2]{AEWZ}
and the telescoping series method imply the observability inequality from two possibly distinct measurable subsets of $\partial\Omega\times (0,T)$ in Theorem \ref{614g3}.
\end{proof}

\section{Observability for second order systems}\label{S:4}

Now, we can apply Lemmas \ref{L: 6086}, \ref{propagation} and the telescoping series method to sketch a proof Theorem \ref{nball}.
\begin{proof}[Proof of Theorem~\ref{nball}]
From Lemma \ref{L: 6086},
\begin{equation*}\label{anas3}
\begin{split}
|\partial_x^\alpha \mathbf{u}(x,L)|
&\leq e^{1/\rho L}|\alpha|!\,\rho^{-|\alpha | }\|\mathbf{u}(0)\|_{L^2(\Omega)^\ell},\;\;\text{for all}\;\;
x\in\overline{\Omega},\;\alpha\in\mathbb{N}^n.
\end{split}
\end{equation*}
Hence, for each $\eta=1,\dots,\ell$, it holds that
\begin{equation*}
|\partial_x^\alpha u^\eta(x,L)|\le
M|\alpha|!\,\rho^{-|\alpha|},\ \text{for all}\ \alpha\in\mathbb{N}^n,\ x\in\overline{\Omega},\ \text{with}\ M=e^{1/\rho L}\|\mathbf{u}(0)\|_{L^2(\Omega)^\ell}.
\end{equation*}
From the propagation of smallness for real-analytic functions from measurable sets (cf. Lemma~\ref{propagation}), we get that for each
$\eta=1,\dots,\ell$, there are $N_\eta=N_\eta(\Omega,\omega_0,\delta)$ and  $\theta_\eta=\theta_\eta(\Omega,\omega_0,\delta)$, $0<\theta_\eta<1$, such that
\begin{equation*}
\|u^\eta(L)\|_{L^2(\Omega)}\leq N_\eta\|u^\eta(L)\|_{L^1(\omega_\eta)}
^{\theta_\eta}M^{1-\theta_\eta}.
 \end{equation*}
Let $N=\max_{1\leq \eta\leq \ell}\{N_\eta\}$ and $\theta=\min_{1\leq \eta\leq \ell}\{\theta_\eta\}$.
Then, we get the following interpolation inequality with $\ell$ different observations:
\begin{equation}\label{difb}
\begin{split}
\|\mathbf{u}(L)\|_{L^2(\Omega)^\ell}&\leq N\Big(\sum_{\eta=1}^{\ell}\|u^\eta(L)
\|^{\theta}_{L^1(\omega_\eta)}\Big)
M^{1-\theta}\\
&\leq N\Big(\sum_{\eta=1}^{\ell}\|u^\eta(L)
\|_{L^1(\omega_\eta)}\Big)
^{\theta}\left(Ne^{N/L}\|\mathbf{u}(0)\|
_{L^2(\Omega)^\ell}\right)^{1-\theta}.
\end{split}
\end{equation}
Next, let $q\in(0,1)$ be a
constant to be determined later and $l$ be a Lebesgue point of $E$. Then, by \cite[Lemma 2]{AEWZ} there is a decreasing sequence  $\{l_m\}_{m\geq 1}$ satisfying $\lim_{m\to\infty} l_m=l$, $l<l_1\le T$ and \eqref{6161can1}. Define as before for each $m\in\mathbb{N}$,
\begin{equation*}
\tau_m=l_{m+1}+(l_{m}-l_{m+1})/6.
\end{equation*}
Then, by the  decay of the energy of the solutions $\mathbf{u}$ to \eqref{anas},
\begin{equation}\ll{unique4}
\|\mathbf{u}(l_m)\|_{L^2(\Omega)^\ell}\leq
\|\mathbf{u}(t)\|_{L^2(\Omega)^\ell},\ \text{for all}\  t\in(\tau_m,l_{m}).
\end{equation}
Moreover, it follows from \eqref{difb} that
\begin{equation*}
\|\mathbf{u}(t)\|_{L^2(\Omega)^\ell}\leq \Big(N
e^{\frac{N}{l_{m}-l_{m+1}}}\sum_{\eta=1}^{\ell}
\|u^\eta(t)
\|_{L^1(\omega_\eta)}
\Big)^{\theta}\|\mathbf{u}(l_{m+1})\|_{L^2(\Omega)^\ell}
^{1-\theta},\ \text{for}\
\tau_m\le t< l_{m}.
\end{equation*}
Applying the Young inequality, we get that for each $\e>0$,
\begin{equation*}
  \|\mathbf{u}(t)\|_{L^2(\Omega)^\ell}\leq \e\|\mathbf{u}(l_{m+1})\|_{L^2(\Omega)^\ell}+
  \e^{-\frac{1-\theta}{\theta}}N
  e^{\frac{N}{l_{m}-l_{m+1}}}\sum_{\eta=1}^
  {\ell}\|u^\eta(t)
\|_{L^1(\omega_\eta)},
\end{equation*}
for $\tau_m\le t <l_{m}$. Integrating the above inequality  over $(\tau_m,l_{m})\cap E$, we have by (\ref{unique4}) that for each $\e>0$,
\begin{equation*}
\begin{split}
\|\mathbf{u}(l_{m})\|_{L^2(\Omega)^\ell}&\leq \e\|\mathbf{u}(l_{m+1})\|_{L^2(\Omega)^\ell}\\
&\;\;\;\;\;\;+\e^{-\frac{1-\theta}{\theta}}
  Ne^{\frac{N}{l_{m}-l_{m+1}}}\int^{l_{m}}_{l_{m+1}}\chi_{E}
  \sum_{\eta=1}^{\ell}\|u^\eta(t)
\|_{L^1(\omega_\eta)}\;dt.
\end{split}
\end{equation*}
Multiplying  the above inequality by $\e^{\frac{1-\theta}\theta}
e^{-\frac{N}{l_m-l_{m+1}}}$ and  replacing $\e$ by $\e^\theta$,
we get
\begin{equation*}
\begin{split}
\e^{1-\theta}e^{-\frac{N}{l_{m}-l_{m+1}}}\|
\mathbf{u}(l_{m})\|_{L^2(\Omega)^\ell}&\leq \e e^{-\frac{N}{l_{m}-l_{m+1}}}\|\mathbf{u}(l_{m+1})
  \|_{L^2(\Omega)^\ell}\\
  &\;\;\;\;\;\;+N\int^{l_{m}}
  _{l_{m+1}}\chi_{E}\sum_{\eta=1}^{\ell}\|u^\eta(t)
\|_{L^1(\omega_\eta)}dt.
\end{split}
\end{equation*}
Choosse then $\e=e^{-\frac{1}{l_{m}-l_{m+1}}}$ to obtain that
\begin{equation}\label{c2}
\begin{split}
  &e^{-\frac{N+1-\theta}{l_{m}-l_{m+1}}}\|\mathbf{u}(l_{m})
  \|_{L^2(\Omega)^\ell}-  e^{-\frac{N+1}{l_{m}-l_{m+1}}}\|\mathbf{u}(l_{m+1})\|
  _{L^2(\Omega)^\ell}\\
  &\;\;\;\;\leq N\int^{l_{m}}_{l_{m+1}}\chi_{E}
 \sum_{\eta=1}^{\ell}\|u^\eta(t)
\|_{L^1(\omega_\eta)}dt,\quad \text{when}\ m\ge 0.
\end{split}
\end{equation}
Finally, we take $q=\frac{N+1-\theta}{N+1}$. Clearly, $0<q<1$ and from \eqref{c2} and \eqref{6161can1}
\begin{equation}\label{10}
\begin{split}
&e^{-\frac{N+1-\theta}{l_{m}-l_{m+1}}}\|\mathbf{u}(l_{m})
  \|_{L^2(\Omega)^\ell}-  e^{-\frac{N+1-\theta}{l_{m+1}-l_{m+2}}}\|\mathbf{u}(l_{m+1})\|_
  {L^2(\Omega)^\ell}\\
&\leq N\int^{l_{m}}_{l_{m+1}}\chi_{E}
  \sum_{\eta=1}^{\ell}\|u^\eta(t)
\|_{L^1(\omega_\eta)}dt\, .
\end{split}
\end{equation}
Summing (\ref{10}) from $m=1$ to $+\infty$ completes the proof.
\end{proof}

Because the full boundary $\partial\Omega$ is analytic, we can use the global internal null controllability for the system \eqref{anas} (a known consequence of Theorem~\ref{614g1} by duality)  and the standard extension method (cf. \cite[p. 249]{AE}) to get the following boundary null controllability: for each $\mathbf{u}_0$ in $L^2(\Omega)^\ell$, there is $\mathbf{g}\in L^2(\p\Omega\times(0,T))^\ell$, with
$$\|\mathbf{g}\|_{L^2(\p\Omega\times(0,T))^\ell}\leq Ne^{N/T}\|\mathbf{u}_0\|
_{L^2(\Omega)^\ell},$$
such that the solution $\mathbf u$ to
\begin{equation*}
\begin{cases}
\partial_t\mathbf u-\mathbf L^{\ast}\mathbf u=0,\ &\text{in}\ \Omega\times(0,T),\\
\mathbf{u}=\mathbf{g},\ &\text{on}\ \partial\Omega\times(0,T),\\
\mathbf{u}(0)=\mathbf{u}_0,\ &\text{in}\ \Omega.
\end{cases}
\end{equation*}
verifies $\mathbf{u}(T)=0$. Also, by the standard duality argument \cite{Lions1}, this boundary null controllability in turn implies the  observability inequality:
\begin{equation*}
\|\mathbf{w}(0)\|_{L^2(\Omega)^\ell}\leq Ne^{N/T}
\|\tfrac{\p\mathbf{w}}{\p\mathbf{\nu}}\|
_{L^2(\partial\Omega\times(0,T))},\quad \left(\tfrac{\p\mathbf{w}}{\p\mathbf{\nu}}\right)^\xi=
a_{ij}^{\xi\eta}\p_{x_j}w^\eta \nu_i,\;\;\;\xi=1,\dots,\ell,
\end{equation*}
for all solutions $\mathbf{w}$ to
\begin{equation*}\label{60915}
\begin{cases}
\partial_t\mathbf w+\mathbf L\mathbf w=0,\ &\text{in}\ \Omega\times(0,T),\\
\mathbf{w}=0,\ &\text{on}
\ \partial\Omega\times(0,T),\\
\mathbf{w}(T)=\mathbf{w}_T,
\ &\text{in}\ \Omega.
\end{cases}
\end{equation*}
with $\mathbf{w}_T$ in $L^2(\Omega)^\ell$. Hence, from the latter and the local energy bound for the system \eqref{anas},
we can derive the following.
\begin{lemma}\label{6101}
There is  $N=N(\Omega,\varrho,\delta)\geq1$ such that the inequality
\begin{equation*}\label{6102}
\begin{split}
&\|\mathbf{u}(T)\|_{L^2(\Omega)^\ell}\leq \left(e^{N/[\left(\epsilon_2-\epsilon_1\right)T]}
\|\tfrac{\p\mathbf{u}}{\p\mathbf{\nu}}\|
_{L^2(\partial\Omega\times(\epsilon_1T,\epsilon_2T))}\right)^{1/2}
\|\mathbf{u}_0\|
_{L^2(\Omega)^\ell}^{1/2},
\end{split}
\end{equation*}
holds for any $0\leq\epsilon_1<\epsilon_2\leq1$ and for all solutions $\mathbf u$ to \eqref{anas}.
\end{lemma}

The Lemmas \ref{L: 6086} and \ref{6101} imply now with similar reasonings to the ones we used in \cite[Theorem 11]{AEWZ}, in the proof of Lemma \ref{6089}, as well as in the proofs of Theorem \ref{614g3} and Theorem \ref{nball}, that Theorems \ref{6090} and \ref{614g2} hold.

To prove Theorem~\ref{bangbang-onecontrol} we need to complete first the proof of Lemma \ref{higherderivative}. With this purpose, we begin with the following lemma.
\begin{lemma}\label{L:1}
Let $f:[0,1]\rightarrow \mathbb{R}$ be  an analytic function verifying
\begin{equation}\label{E:1}
\|f^{(m)}\|_{L^\infty(0,1)}\leq M\rho^{-m}m!,\ \text{when}\  m\geq0,
\end{equation}
for some $M>0$ and $0<\rho\leq 1/2$. Then
\begin{equation}\label{9241}
\|f^{(j)}\|_{L^\infty(0,1)}\leq
\big(8M(j+1)!\rho^{-j-1}\big)^{1-\frac{1}{2^j}}
\|f\|_{L^\infty(0,1)}^{\frac{1}{2^j}}\,,\;\;\;\text{when}\;\;
j\geq 0.
\end{equation}
\end{lemma}
\begin{proof}
We prove it by induction and we assume that \eqref{9241}
holds for $(k-1)$, i.e.,
\begin{equation}\label{assumption}
\|f^{(k-1)}\|_{L^\infty(0,1)}\leq (8Mk!\rho^{-k})^{1-\frac{1}{2^{k-1}}}
\|f\|_{L^\infty(0,1)}^{\frac{1}{2^{k-1}}}
\end{equation}
and we  show that it is valid for $k$.
Let then $x\in[0,1]$. For $0<\varepsilon\leq1/2$ take either  $I=[x,x+\varepsilon]$ or $[x-\varepsilon,x]$, so that always $I\subset [0,1]$. Then,
\begin{equation*}
f^{(k)}(x)=f^{(k)}(y)+\int_y^x f^{(k+1)}(s)\,ds,\;\;\text{for all}\;\;y\in I.
\end{equation*}
Integrating the above identity with respect to $y$ over the interval $I$,  by \eqref{E:1} and the arbitrariness of $x$ in $[0,1]$,
we obtain that
\begin{equation}\label{l1}
\|f^{(k)}\|_{L^\infty(0,1)}\leq \varepsilon M(k+1)!\rho^{-k-1}+\frac{2}{\varepsilon}\,
\|f^{(k-1)}\|_{L^\infty(0,1)},
\end{equation}
when $k\geq1$ and $0<\varepsilon\leq1/2$. Choose now
$$\varepsilon=\Big(\frac{2\|f^{(k-1)}\|_{L^\infty(0,1)}}
{M(k+1)!\rho^{-k-1}}\Big)^{1/2}.$$
It can be checked by \eqref{E:1} that $\varepsilon\leq1/2$.
Hence, it follows from \eqref{l1} that
\begin{align*}
\|f^{(k)}\|_{L^\infty(0,1)}&\leq
\big(8M(k+1)!\rho^{-k-1}\big)^{1/2}
\|f^{(k-1)}\|_{L^\infty(0,1)}^{1/2}.
\end{align*}
This, together with \eqref{assumption}, leads to \eqref{9241} and
completes the proof.
\end{proof}

The rescaled and translated version of Lemma~\ref{L:1},
along with \cite[Lemma 13]{AEWZ},
imply the following.
\begin{lemma}\label{982}
Let $f$ be real-analytic in $[a,a+L]$ with $a$ in $\mathbb{R}$, $L>0$ and $E\subset[a,a+L]$ be a measurable set with positive measure. Assume there are constants $M>0$ and $0<\rho\leq1/2$ such that
\begin{equation*}
|f^{(m)}(x)|\leq M (2\rho L)^{-m}m!,\ \text{for}\ m\geq 0\ \text{and}\ a\le x\le a+L.
\end{equation*}
Then, there are $N=N(\rho,|E|/L)$ and $\theta=\theta(\rho, |E|/L)$ with $0<\theta<1$, such that
\begin{equation*}
\|f^{(k)}\|_{L^\infty(a,a+L)}\leq N\big(8(k+1)!(\rho L)^{-(k+1)}\big)
M^{1-\frac{\theta}{2^{k}}}\Big(
\text{\rlap|{$\int_{E}$}}|f|\,dx\Big)
^{\frac{\theta}{2^{k}}}\,,\;\;\text{when}\;\;k\geq0.
\end{equation*}
\end{lemma}
\bigskip

Next, we derive the multi-dimensional analogs of Lemmas~\ref{L:1} and \ref{982}.

\begin{lemma}\label{L:2} Let $n\ge 1$ and $f:Q\subset\mathbb{R}^n\rightarrow \mathbb{R}$, with $Q=[0,1]
\times\cdots\times[0,1]$, be a real-analytic function verifying
\begin{equation}\label{9222}
\|\partial_{x_1}^{\beta_1}\cdots\partial_{x_n}^{\beta_n}f\|
_{L^\infty(Q)}\leq M\rho^{-|\beta|}\beta_1!\cdots\beta_n!,\;\;
\forall \beta=(\beta_1,\dots,\beta_n)\in\mathbb{N}^n,
\end{equation}
for some  $M>0$ and $0<\rho\leq1/2$. Then,
\begin{equation}\label{9225}
\|\partial_{x_1}^{\alpha_1}\cdots\partial_{x_n}^{\alpha_n}
f\|_{L^\infty(Q)}\leq \Big(8M\rho^{-|\alpha|-1}\prod_{i=1}^{n}(\alpha_i+1)!\Big)
^{1-\frac{1}{2^{|\alpha|}}}\|f\|_{L^\infty(Q)}^{
\frac{1}{2^{|\alpha|}}}.
\end{equation}
holds for each $\alpha=(\alpha_1,\dots,\alpha_n)\in\mathbb{N}^n$.
\end{lemma}
\begin{proof}
First, notice that Lemma~\ref{L:1} corresponds to Lemma \ref{L:2}, when $n=1$.  Let now $n\geq2$ and $\alpha=(\alpha_1,\dots,\alpha_n)$ be in $\mathbb{N}^n$. For $(x_1,\dots,x_{n-1})$ in $[0,1]\times\cdots\times[0,1]$, define
the function $g_n:[0,1]\rightarrow\mathbb{R}$ by
$$g_n(x_n)\triangleq\partial_{x_1}^{\alpha_1}\cdots
\partial_{x_{n-1}}^{\alpha_{n-1}}f(x_1,\cdots,x_{n-1},x_{n}).$$
It follows from \eqref{9222} that
\begin{equation*}
\|\partial_{x_n}^{\beta_n}g_n\|_{L^\infty([0,1])}\leq \Big(M\alpha_1!\cdots\alpha_{n-1}!\rho^{-\sum_{j=1}^{n-1}\alpha_j}\Big)
\beta_n!\rho^{-\beta_n},\ \text{for all}\  \beta_n\geq0,
\end{equation*}
and Lemma~\ref{L:1}  yields that
\begin{multline*}
\|\partial_{x_1}^{\alpha_1}\cdots\partial_{x_n}^{\alpha_n}
f\|_{L^\infty(Q)}\\
\;\leq
\Big(8M\alpha_1!\cdots\alpha_{n-1}!\rho^{-\sum_{j=1}^{n-1}\alpha_j}(\alpha_n+1)!\rho^{-\alpha_n-1}\Big)
^{1-\frac{1}{2^{\alpha_n}}}\|
\partial_{x_1}^{\alpha_1}\cdots\partial_{x_{n-1}}^{\alpha_
{n-1}}f\|_{L^\infty(Q)}^{\frac{1}{2^{\alpha_{n}}}}.
\end{multline*}
Similarly, we can show that $\|\partial_{x_1}^{\alpha_1}\cdots\partial_{x_{n-1}}^{\alpha_{n-1}}
f\|_{L^\infty(Q)}$ is less or equal than
\begin{equation*}
\Big(8M\alpha_1!\cdots\alpha_{n-2}!\rho^{-\sum_{j=1}^{n-2}\alpha_j}(\alpha_{n-1}+1)!\rho^{-\alpha_{n-1}-1}\Big)
^{1-\frac{1}{2^{\alpha_{n-1}}}}\|
\partial_{x_1}^{\alpha_1}\cdots\partial_{x_{n-2}}^{\alpha_
{n-2}}f\|_{L^\infty(Q)}^{\frac{1}{2^{\alpha_{n-1}}}}.
\end{equation*}
The iteration of the above arguments $n$ times leads to the desired estimates in \eqref{9225}.
\end{proof}

The rescaled  and translated versions of Lemma~\ref{L:2} and of Lemma~\ref{propagation} (when $\Omega$ is the unit ball or cube in $\Rn$) and the fact that a ball in $\mathbb{R}^n$ contains a cube of comparable diameter and vice versa are seen to impy Lemma \ref{higherderivative} .

Finally, we give the proof of Theorem~\ref{bangbang-onecontrol}, where we use Lemma~\ref{982} with $k=1$ and Lemma \ref{higherderivative} with $|\alpha|\le 2$.

\begin{proof}[Proof of Theorem~\ref{bangbang-onecontrol}]
Since $b(\cdot)\not\equiv0$ in $\Omega$ and $b$ is real-analytic in $\overline\Omega$, we may assume without loss of generality, that $|b(x)|\geq1$ over some ball $B_{R}(x_0)\subset\Omega$ and that $\mathcal{D}\subset B_R(x_0)\times(0,T)$.
By Lemma \ref{L: 6086}, for $x$ in $\overline\Omega$ and $0\le s<t$,
\begin{multline}\label{jingcan1}
|\partial_x^\alpha\partial_t^p u(x,t)|+
|\partial_x^\alpha\partial_t^p v(x,t)|\\
\le e^{1/\rho\left(t-s\right)} |\alpha|!\,p!\,\rho^{-|\alpha |-p}\left(t-s\right)^{-p}
\big[\|u(s)\|_{L^2(\Omega)}+\|v(s)\|
_{L^2(\Omega)}\big],
\end{multline}
for all $\alpha\in\mathbb{N}^n$ and $p\in\mathbb{N}$, with $\rho=\rho(\delta)$, $0<\rho\le 1$.
Hence, we can get  from \eqref{difb} that
\begin{multline*}
\|u(t)\|_{L^2(\Omega)}+\|v(t)\|
_{L^2(\Omega)}
\leq \\
\Big(\int_{B_R(x_0)}|u(x,t)|+|v(x,t)|\,dx
\Big)^\theta
\Big(Ne^{N/(t-s)}\big(\|u(s)\|_{L^2(\Omega)}+\|v(s)\|
_{L^2(\Omega)}\big)\Big)^{1-\theta},
\end{multline*}
with $N=N(\Omega,\rho,R)$ and $\theta=\theta(\Omega,\rho,R)$, $0<\theta<1$.
This, together with the fact that $|b(x)|\geq1$ over $B_R(x_0)$
and the first equation in \eqref{heats2}, yield that
\begin{multline}\label{jingcan8}
\|u(t)\|_{L^2(\Omega)}+\|v(t)\|
_{L^2(\Omega)}\\
\leq
\Big(\int_{B_R(x_0)}|u(x,t)|+|\partial_t u(x,t)|
+|\Delta u(x,t)|\,dx
\Big)^\theta\\
\;\;\;\;\;\;\;\times
\Big(Ne^{N/(t-s)}\big(\|u(s)\|_{L^2(\Omega)}+\|v(s)\|
_{L^2(\Omega)}\big)\Big)^{1-\theta},\ \text{when}\ 0\leq s<t.
\end{multline}
Next, let $\eta\in(0,1)$ and $0\leq t_1<t_2$. Also, assume that $E\subset(0,T)$ is a measurable set
with $|E\cap(t_1,t_2)|\geq \eta(t_2-t_1)$, for some $\eta\in (0,1)$, and that for each $t\in E$, $|\mathcal{D}_t|\triangleq|\{x\in\Omega: (x,t)\in\mathcal{D}\}|\geq \gamma|\mathcal D|$, for some $\gamma>0$.
Set then
$$\tau=t_1+\frac{\eta}{10}(t_2-t_1)\ \text{and}\ F=[\tau,t_2]\cap E.$$
Clearly, $|F|\geq\frac\eta 2\,(t_2-t_1)$. Hence, it follows from \eqref{jingcan1} that when
$t\in[\tau,t_2]$ and $x$ is in $\Omega$
\begin{equation*}\label{jingcan4}
|\partial_t^p u(x,t)|\leq \frac{p!Ne^{N/\eta(t_2-t_1)}}{(\eta(t_2-t_1)/20)^p}  \big(\|u(t_1)\|_{L^2(\Omega)}+\|v(t_1)\|
_{L^2(\Omega)}\big),\ \text{for all}\ p\in\mathbb{N},
\end{equation*}
with $N=N(\Omega,\rho)$.
By Lemma~\ref{982}, we have that for each $x$ in $\Omega$
\begin{multline*}
\|\partial_t u(x,\cdot)\|_{L^\infty([\tau,t_2])}
\leq \\\Big(\int_{F}|u(x,s)|\,ds\Big)^{\theta}
\Big(Ne^{N/(t_2-t_1)}\big(\|u(t_1)\|_{L^2(\Omega)}+
\|v(t_1)\|
_{L^2(\Omega)}\big)\Big)^{1-\theta},
\end{multline*}
with $N=N(\Omega,\rho,\eta)$ and $\theta=\theta(\Omega,\rho,\eta)$, $0<\theta<1$.
Hence, by H\"{o}lder's inequality
\begin{multline}\label{jingcan5}
\int_{B_R(x_0)}|\partial_t u(x,t)|\,dx
\leq \\
\Big(Ne^{N/(t_2-t_1)}\big(\|u(t_1)\|_{L^2(\Omega)}+
\|v(t_1)\|
_{L^2(\Omega)}\big)\Big)^{1-\theta}
\Big(\int_{F}\int_{B_R(x_0)}|u(x,s)|\,dxds\Big)^{\theta}
\end{multline}
when $\tau\le t\le t_2$. It also follows from \eqref{jingcan1} that when $\tau\le t\le t_2$ and $x$ is in $\Omega$, we have
\begin{equation*}\label{jingcan10}
|\partial_x^\alpha u(x,t)|\le |\alpha|!\rho^{-|\alpha |} N e^{N/\left(t_2-t_1\right)} \big(\|u(s)\|_{L^2(\Omega)}+\|v(s)\|
_{L^2(\Omega)}\big),\  \text{for all}\ \alpha\in\mathbb{N}^n,
\end{equation*}
with $N=N(\Omega,\rho,\eta)$.
 Now, it holds that for each $t\in F$, $|\mathcal{D}_t|\geq \gamma |\mathcal D|$, and it follows from Theorem~\ref{higherderivative} that
\begin{multline}\label{jingcan6}
\int_{B_R(x_0)}|u(x,t)|\,dx\leq\\
\Big(\int_{\mathcal{D}_t}|u(x,t)|\,dx\Big)^\theta
\Big(Ne^{N/(t_2-t_1)}\big(\|u(t_1)\|_{L^2(\Omega)}+
\|v(t_1)\|
_{L^2(\Omega)}\big)\Big)^{1-\theta}
\end{multline}
and
\begin{multline}\label{jingcan12}
\int_{B_R(x_0)}|\Delta u(x,t)|\,dx\leq\\
\Big(\int_{\mathcal{D}_t}|u(x,t)|\,dx\Big)^\theta
\Big(N e^{N/\left(t_2-t_1\right)} \big(\|u(t_1)\|_{L^2(\Omega)}+\|v(t_1)\|
_{L^2(\Omega)}\big)\Big)^{1-\theta}.
\end{multline}
with $N=N(\Omega,|\mathcal D|,R,\rho,\eta)$ and $\theta=\theta(\Omega,|\mathcal D|,R,\rho,\eta)$, $0<\theta<1$.
Hence, \eqref{jingcan5} and \eqref{jingcan6}, as well as H\"{o}lder's inequality imply that
\begin{multline*}
\int_{B_R(x_0)}|\partial_t u(x,t)|\,dx
\leq\\
 \Big(\int_{t_1}^{t_2}\chi_{E}(s)\|u(s)\|
_{L^1(\mathcal{D}_s)}\,ds\Big)^{\theta}\Big(Ne^{N/(t_2-t_1)}
\big(\|u(t_1)\|_{L^2(\Omega)}+\|v(t_1)\|
_{L^2(\Omega)}\big)\Big)^{1-\theta},
\end{multline*}
when $t\in F$.
This, together with the inequalities \eqref{jingcan8}, \eqref{jingcan6}, \eqref{jingcan12} and H\"{o}lder's inequality,  yield  that the inequality
\begin{multline*}
\|u(t)\|_{L^2(\Omega)}+\|v(t)\|
_{L^2(\Omega)}
\leq
\Big(\int_{t_1}^{t_2}\chi_{E}(s)\|u(s)\|
_{L^1(\mathcal{D}_s)}\,ds+
\int_{\mathcal{D}_t}|u(x,t)|\,dx\Big)^{\theta}\\
\;\;\;\;\times
\Big(N e^{N/\left(t_2-t_1\right)} \|u(t_1)\|_{L^2(\Omega)}+\|v(t_1)\|
_{L^2(\Omega)}\Big)^{1-\theta},
\end{multline*}
holds for $t\in F$.
Integrating the above inequality with respect to time over the set $F$, recalling that $|F|\geq \frac\eta 2(t_2-t_1)$, using the energy estimate  for solutions to the equations \eqref{heats2} and H\"{o}lder's inequality,  we find  that
\begin{multline*}
\|u(t_2)\|_{L^2(\Omega)}+\|v(t_2))\|_{L^2(\Omega)}
\leq\\
 \Big(\int_{t_1}^{t_2}\chi_{E}(t)\|u(t)\|
_{L^1(\mathcal{D}_t)}\,dt\Big)^{\theta}
\Big(N e^{N/\left(t_2-t_1\right)} \big(\|u(t_1)\|_{L^2(\Omega)}+\|v(t_1))\|
_{L^2(\Omega)}\big)\Big)^{1-\theta},
\end{multline*}
with $N=N(\Omega,|\mathcal D|, R,\rho,\eta)$ and $\theta=\theta(\Omega,|\mathcal D|, R,\rho,\eta)$, $0<\theta<1$.

Finally, by Fubini's theorem and following the reasonings within the second part of the proof of Theorem~\ref{6106} (i.e., the telescoping series method)  we can also derive the desired observability estimate in Theorem \ref{bangbang-onecontrol}.
\end{proof}
\section{Applications to control theory}\label{S:5}
In this Section, we show several applications
of Theorems~\ref{6131}, \ref{6106}, \ref{nball} and \ref{bangbang-onecontrol}  in control theory. One can also obtain  analogous  applications of Theorems~\ref{614g3}, \ref{6090}, \ref{614g2} and \ref{614g1}.

First of all, we can apply Theorem~\ref{6131} to get the bang-bang property of the time optimal control problems for the higher
order parabolic equations \eqref{E: 6081}: let $\Omega$ be
a bounded domain with analytic boundary and $\omega\subset\Omega$ a non-empty open set (or a measurable set with positive measure). Define for each $M>0$ a control constraint set
\begin{equation*}
\mathcal{U}_1^M\triangleq\Big\{f:\Omega\times\mathbb{R}^+\rightarrow
\mathbb{R} \;\;\text{measurable} : |f(x,t)|\leq M,
\;\;\text{a.e. in}\ \Omega\times\mathbb{R}^+
\Big\}.
\end{equation*}
For each $u_0$ in $L^2(\Omega)\setminus\{0\}$, consider the time optimal control problem
$$
(TP)_1^M:\;\;\;\;\;\;T^M_1\triangleq \inf_{\mathcal{U}^M_1}
\big\{t>0;\;u(t;u_0,f)=0\big\},
$$
where $u(\cdot\;;u_0,f)$ is the solution to the controlled problem
\begin{equation*}
\begin{cases}
\p_t u+(-1)^m\Delta^m u=\chi_\omega f,\ &\text{in}\ \Omega\times(0,+\infty),\\
u=\nabla u=\dots=\nabla^{m-1}u=0,\ &\text{on}\ \p\Omega\times(0,+\infty),\\
u(0)=u_0.\ &\text{in}\ \Omega.
\end{cases}
\end{equation*}
According to Theorem \ref{6131} and  \cite[Theorem 3.3]{PWZ1}, $T^M_1$ is a positive minimum. A control function $f$ in $\mathcal{U}^M_1$ associated to $T^M_1$ is called an optimal control to this problem. Then, the methods in \cite[\S 5]{AEWZ} (See also \cite{W1} or \cite{PW1}), and the fact that standard duality (HUM method) and Theorem \ref{6131} imply the null controllability at all times $T>0$ of the system \eqref{E: 6081} with bounded controls acting over measurable sets within $\omega\times (0,T)$, give the following result.
\begin{corollary}
Problem $(TP)_1^M$ has the bang-bang property: any time optimal control $f$ satisfies, $|f(x,t)|=M$, for a.e. $(x,t)$ in $\omega\times (0, T_1^M)$. Consequently, the problem has a unique time optimal control.
\end{corollary}

Theorem \ref{6106} implies a weak bang-bang property for the time optimal boundary control problems for the fourth order parabolic equation \eqref{6084}: let $\Omega$ be as above and $\Gamma\subset\p\Omega$ be a non-empty open subset (or a measurable set in $\partial\Omega$ with positive surface measure). Define for each $M>0$ the control constraint set
\begin{equation*}
\begin{split}
\mathcal{U}_2^M\triangleq&\Big\{(g_1,g_2):\p\Omega\times
\mathbb{R}^+\rightarrow
\mathbb{R}^2 \;\;\text{measurable}\;;\;\\
&\;\;\;\;\max\big\{|g_1(x,t)|,|g_2(x,t)|\big\}\leq M,\;\;\text{a.e.}\;\; (x,t)\in \Omega\times\mathbb{R}^+
\Big\}.
\end{split}
\end{equation*}
For each $u_0$ in $L^2(\Omega)\setminus\{0\}$ consider the time optimal boundary control problem
\begin{equation*}
(TP)_2^M:\quad T^M_2\triangleq \inf_{\mathcal{U}^M_2}
\big\{t>0;\;u(t;u_0,g_1,g_2)=0\big\},
\end{equation*}
where $u(\cdot\;;u_0,g_1,g_2)$ denotes the solution to the boundary controlled parabolic equation
\begin{equation}\label{E: unprobelamjo}
\begin{cases}
\p_t u+\Delta^2 u=0,\ &\text{in}\ \Omega\times(0,T),\\
u=g_1 \chi_{\Gamma},\;\; \frac{\p u}{\p\nu}=g       _2 \chi_{\Gamma}, &\text{on}\ \p\Omega\times(0,T),\\
u(0)=u_0,\ &\text{in}\ \Omega.
\end{cases}
\end{equation}
From Theorem~\ref{6106} and arguments as those in the proof of \cite[Lemma 15]{AEWZ}, $T^M_2$  is a positive minimum. A control pair of functions $(g_1,g_2)$ associated to $T^M_2$ is called an optimal control to this problem. From Theorem \ref{6106} and similar
methods to those in \cite[\S 5]{AEWZ}, give the following non-standard bang-bang property:
\begin{corollary}\label{C: colotati}
Problem $(TP)_2^M$ has the weak bang-bang property: any time optimal control $(g_1,g_2)$ satisfies that $\max\big\{|g_1(x,t)|,|g_2(x,t)|\big\}=M$, for a.e. $(x,t)$ in $\Gamma\times (0, T_2^M)$.
\end{corollary}
To carry out the technical details for Corollary \ref{C: colotati} we must  first solve \eqref{E: unprobelamjo} for $u_0$ in $L^2(\Omega)$ and with lateral Dirichlet data $g_i$, $i=1,2$, in $L^\infty(\partial\Omega\times (0,T))$. For this reason by the solution to
\begin{equation}\label{E: unprobelamjo2}
\begin{cases}
\p_t u+\Delta^2 u=0,\ &\text{in}\ \Omega\times(0,T),\\
u=g_1,\ \frac{\p u}{\p\nu}=g_2, &\text{on}\ \p\Omega\times(0,T),\\
u(0)=u_0,\ &\text{in}\ \Omega,
\end{cases}
\end{equation}
with $g_i$, $i=1,2$, in $L^2(\partial\Omega\times (0,T))$ and $u_0$ in $ C_0^\infty(\Omega)$, we mean the unique function $u$ over $\Omega\times (0,T)$ such that $v=u-e^{t(-\Delta^2)}u_0$ is the solution defined by transposition \cite[p. 209]{Lions1} to
\begin{equation}\label{tonteria}
\begin{cases}
\p_t v+\Delta^2 v=0,\ &\text{in}\ \Omega\times(0,T),\\
v=g_1,\ \frac{\p v}{\p\nu}=g_2, &\text{on}\ \p\Omega\times(0,T),\\
v(0)=0,\ &\text{in}\ \Omega,
\end{cases}
\end{equation}
i.e.; the unique $v$ in $L^2(\Omega\times (0,T))$ verifying
\begin{equation*}
\int_{\Omega\times (0,T)}v\left(-	 \partial_t\varphi+\Delta^2\varphi\right)\,dxdt=\int_{\partial\Omega\times (0,T)}g_1\tfrac{\partial\Delta\varphi}{\partial\nu}-g_2\,\Delta\varphi\,d\s dt,
\end{equation*}
for all $\varphi$ in $C^\infty(\overline\Omega\times [0,T])$, with $\varphi(T)\equiv 0$ in $\Omega$ and $\varphi=\nabla\varphi=0$ in $\partial\Omega\times (0,T)$. One can make sense of $v$ because for each $h$ in $C^\infty(\overline\Omega\times [0,T))$ there is a unique $\varphi$ in $C^\infty(\overline\Omega\times [0,T])$ verifying
\begin{equation*}\label{E: ecuaciadjunta}
\begin{cases}
-\partial_t\varphi+\Delta^2\varphi= h,\ &\text{in}\ \Omega\times (0,T),\\
\varphi=\frac{\partial\varphi}{\partial\nu}=0,\ &\text{in}\ \partial\Omega\times (0,T),\\
\varphi(T)=0,\ &\text{in}\ \Omega,
\end{cases}
\end{equation*}
and
\begin{equation*}
\|\varphi\|_{L^\infty((0,T), L^2(\Omega))\cap L^2((0,T), H^4(\Omega)\cap H^2_0(\Omega))}\le N\|h\|_{L^2(\Omega\times (0,T))},
\end{equation*}
with $N=N(\Omega,T)$ \cite[p. 140, Theorem 10.2]{Friedman1}. The above estimate on $\varphi$, standard traces inequalities \cite[p. 258]{Evans1} and duality imply the bound
\begin{equation}\label{E: acotac}
\|v\|_{L^2(\Omega\times (0,T))}\le N\left[\|g_1\|_{L^2(\partial\Omega\times (0,T))}+\|g_2\|_{L^2(\partial\Omega\times (0,T))}\right],
\end{equation}
with $N$ as above. For given $T>0$, $u_0$ in $L^2(\Omega)$ and $\mathcal J\subset\partial\Omega\times (0,T)$, a measurable set with positive measure, we may assume that $\mathcal J\subset \Omega\times (0,T-2\delta)$ for some  small $0<\delta<T/2$. Then, the existence of  two bounded boundary control functions $g_i$, $i=1,2$, verifying
\begin{equation*}
\|g_1\|_{L^\infty(\mathcal J)} + \|g_2\|_{L^\infty(\mathcal J)}\le N\|u_0\|_{L^2(\Omega)},
\end{equation*}
with $N$ the constant in \eqref{E: firatobservanbi} for the new set $\mathcal J$ and
such that the solution $u$ to
\begin{equation}\label{E: conroloed}
\begin{cases}
\p_t u+\Delta^2 u=0,\ &\text{in}\ \Omega\times(0,T),\\
u=g_1 \chi_{\mathcal J},\;\; \frac{\p u}{\p\nu}=g_2 \chi_{\mathcal J}, &\text{on}\ \p\Omega\times(0,T),\\
u(0)=u_0,\ &\text{in}\ \Omega.
\end{cases}
\end{equation}
verifies $u(T)\equiv 0$, can be proved by means of a standard  duality argument (Hahn Banach Theorem) based on the observability inequality \eqref{E: firatobservanbi} \cite[Corollary 1]{AEWZ} with the purpose to obtain the existence of two functions $g_i$ in $L^\infty(\mathcal J)$, $i=1,2$, verifying
\begin{equation}\label{E: controles}
\int_{\Omega}u_0\varphi(0)\,dx +\int_{\mathcal J}g_1\tfrac{\partial\Delta\varphi}{\partial\nu}-g_2\,\Delta \varphi\,d\s dt=0,
\end{equation}
for all $\varphi_T$ in $C_0^\infty(\Omega)$ and with $\varphi(t)=e^{\left(t-T\right)\Delta^2}\varphi_T$. Recall that the unique weak solution $v$ to \eqref{tonteria} is in fact in $C^\infty(\overline\Omega\times [0, +\infty))$, when $g_i$  are both in $C_0^\infty(\partial\Omega\times [0,+\infty))$; and that there is $N=N(\Omega,\delta)$ such that the estimate
\begin{equation*}
\|v\|_{C^{2,1}(\Omega\times [T-\frac\delta 2, T])}\le N\|v\|_{L^2(\Omega\times (0,T))},
\end{equation*}
holds when $\text{supp}(g_i)\subset \partial\Omega\times [0,T-\delta]$, $i=1,2$ \cite[p.141]{Friedman1}. The latter and \eqref{E: acotac} yield the bound
\begin{equation}\label{E: acotac2}
\|v\|_{C^{2,1}(\Omega\times[T-\frac\delta 2, T])}\le N(\Omega, T,\delta)\left[\|g_1\|_{L^2(\partial\Omega\times (0,T))}+\|g_2\|_{L^2(\partial\Omega\times (0,T))}\right],
\end{equation}
when $\text{supp}(g_i)\subset \partial\Omega\times [0,T-\delta]$, $i=1,2$. Finally, letting $u^\e$ denote the $C^\infty(\overline\Omega\times [0,+\infty))$ solution to \eqref{E: unprobelamjo2}, when $u_0$ and $g_i$ are replaced respectively by $u_0^\e$ and $g_i^\e$, with $u_0^\e$ in $C_0^\infty(\Omega)$, $g_i^\e$ in $C_0^\infty(\partial\Omega\times (0,T-\delta))$ for $i=1,2$, and in such a way that $u_0^\e$ converges to $u_0$ in $L^2(\Omega)$ and $g_i^\e$ converges to $g_i\chi_{\mathcal J}$ in $L^2(\partial\Omega\times (0,T-\delta))$, with
\begin{equation*}
\|g_i^\e\|_{L^\infty(\partial\Omega\times [0,T-\delta])}\le 2\|g_i\|_{L^\infty(\mathcal J)},\ \text{for}\  i=1,2,
\end{equation*}
integration by parts shows that
\begin{equation*}
\int_\Omega u^\e(T)\varphi_T\,dx =\int_\Omega u^\e_0\varphi(0)\,dx+\int_{\mathcal J}g_1^\e\,\tfrac{\partial\Delta\varphi}{\partial\nu}-g_2^\e\,\Delta \varphi\,d\s dt,
\end{equation*}
when $\varphi=e^{\left(t-T\right)\Delta^2}\varphi_T$, $\varphi_T$ is in $C^\infty_0(\Omega)$. Letting then $\e\to 0^+$ together with \eqref{E: acotac2} and \eqref{E: controles} show that the solution $u$ to \eqref{E: conroloed} verifies $u\equiv 0$ for $t\ge T$. The proof of Corollary \ref{C: colotati} is now standard.

Theorem \ref{nball} implies the null controllability of the system \eqref{anas} with controls restricted over $\ell$ different
non-empty open sets (or  measurable sets of positive measure): assume that $\omega_j\subset\Omega$, $j=1,\dots,\ell$, are
non-empty open sets verifying, $\omega_j\cap\omega_k=\emptyset$, for $1\leq j\neq k\leq \ell$.  Consider the system
\begin{equation}\label{6151}
\begin{cases}
\partial_t\mathbf u-\mathbf L^*\mathbf u=\mathbf f,
 &\text{in}\ \Omega\times(0,T),\\
\mathbf{u}=0,\ &\text{on}\ \partial\Omega\times(0,T),\\
\mathbf{u}(0)=\mathbf{u}_0,\ &\text{in}\ \Omega,
\end{cases}\quad\quad \text{with}\ \mathbf f=\left(\chi_{\omega_1}f_1,\dots, \chi_{\omega_l}f_l\right),
\end{equation}
$f_\xi$ in $L^\infty(\Omega\times(0,T))$, $\xi=1,\dots,\ell$, are the
controls, $\mathbf{u}_0$ in $L^2(\Omega)^\ell$ and $\mathbf L$ and its coefficients are as in \eqref{anas}. Then, from Theorem \ref{nball} and the classical duality argument (cf., e.g., \cite{DRu} or \cite[Corollary 1]{AEWZ}) we have
\begin{corollary}
For each $T>0$ and $\mathbf{u}_0$ in $L^2(\Omega)^\ell$, there are
bounded controls  $\mathbf{f}=(f_1,\dots,f_\ell)$, with
$$\|\mathbf{f}\|_{L^\infty(\Omega\times(0,T))}\leq N\|
\mathbf{u}_0\|_{L^2(\Omega)^\ell},$$
such that the solution $\mathbf{u}(\cdot\;;\mathbf{u}_0,\mathbf{f})$ to \eqref{6151} verifies, $\mathbf{u}(T;\mathbf{u}_0,\mathbf{f})=0$. Here, the constant $N=N(T,\Omega,\omega_1,\dots,\omega_\ell)$ is independent of
$\mathbf{u}_0$.
\end{corollary}

Finally, Theorem~\ref{bangbang-onecontrol} implies the bang-bang
property of the time optimal controls for some systems of two parabolic equations with only one control force. For this connection we refer the readers to
\cite{ABDK}, \cite{WL} and the references therein: let
$T>0$ and $\Omega$ be a as above. Suppose that $a(\cdot)$, $b(\cdot)$, $c(\cdot)$ and $d(\cdot)$ are real-analytic in $\overline{\Omega}$ and $b(\cdot)\not\equiv0$. Let $\omega\subset \Omega$ be a non-empty open set (or a measurable set with positive measure).
Consider the controlled parabolic system
\begin{equation}\label{925c1}
\begin{cases}
\partial_t u-\Delta u+a(x)u+b(x)v=0,\ &\text{in}\;\;\Omega\times(0,+\infty),\\
\partial_t v-\Delta v+c(x)u+d(x)v=\chi_{\omega}f,\ &\text{in}\;\;\Omega\times(0,+\infty),\\
u=0,\;\;v=0,\;\;&\text{on}\;\;\partial\Omega\times(0,+\infty),\\
u(\cdot,0)=u_0,\;\;v(\cdot,0)=v_0,\;\;&\text{in}\;\;\Omega,
\end{cases}
\end{equation}
where $f$ is a control force taken in the constraint set
\begin{equation*}
\mathcal{U}_3^M\triangleq\Big\{f:\Omega\times\mathbb{R}^+\rightarrow
\mathbb{R} \;\;\text{measurable} : |f(x,t)|\leq M,
\;\;\text{a.e. in}\ \Omega\times\mathbb{R}^+
\Big\},
\end{equation*}
with $M>0$. For each $(u_0,v_0)$ in $L^2(\Omega)\times L^2(\Omega)\setminus\{(0,0)\}$, we study the time optimal control problem
$$
(TP)_3^M:\;\;\;\;\;\;T^M_3\triangleq \inf_{\mathcal{U}^M_3}
\big\{t>0;\;\big(u(t;u_0,v_0,f),v(t;u_0,v_0,f)\big)=(0,0)\big\},
$$
where $\big(u(\cdot\;;u_0,v_0,f),v(\;;,u_0,v_0,f)\big)$ is the solution to \eqref{925c1} corresponding to the control $f$ and the initial datum $(u_0,v_0)$. Then, the
methods in \cite[\S 5]{AEWZ} and Theorem \ref{bangbang-onecontrol} give the following consequence.
\begin{corollary}
The problem $(TP)_3^M$ has the bang-bang property: any time optimal control $f$ satisfies, $|f(x,t)|=M$ for a.e. $(x,t)$ in $\omega\times (0, T_3^M)$. Moreover, it is unique.
\end{corollary}
\noindent {\it Acknowledgement}: The authors wish to thank Professor Gengsheng Wang for his suggestions during the writing of this work.

\end{document}